\documentclass[10pt]{amsart}

\usepackage{geometry,graphicx,amssymb,amsmath,amsbsy,eucal,amsfonts,mathrsfs,amscd,bm,tcolorbox,amsthm}

\usepackage{tikz}
\usepackage{subcaption}
\numberwithin{equation}{section}

\allowdisplaybreaks[3]

\newtheorem{theorem}{Theorem}[section]
\newtheorem{lemma}[theorem]{Lemma}
\newtheorem{assumption}[theorem]{Assumption}
\newtheorem{corollary}[theorem]{Corollary}
\newtheorem{proposition}[theorem]{Proposition}
\theoremstyle{definition}

\theoremstyle{remark}
\newtheorem{remark}[theorem]{Remark}
\newtheorem{example}[theorem]{Example}

\newcommand{\p}{{\partial}}

\newcommand{\bl}{\bigl\langle}
\newcommand{\br}{\bigr\rangle}

\newcommand{\bld}[1]{\boldsymbol{#1}}

\newcommand{\bn}{\bld{n}}

\newcommand{\Dt}{\Delta t}
\newcommand{\uu}{\mathsf{u}}
\newcommand{\ww}{\mathsf{w}}

\newcommand{\lla}{\mathsf{l}}

\newcommand{\Lla}{\Lambda}
\newcommand{\LL}{\mathsf{L}}

\newcommand{\pdt}{\partial_{\Delta t}}
\newcommand{\pt}{\partial_{t}}
\newcommand{\dt}{\Delta t}
\newcommand{\Z}{\mathsf{Z}}
\newcommand{\Ss}{\mathsf{S}}
\newcommand{\J}{\mathsf{J}}

\title[Correction Methods for Interface Problems]{A second-order correction method for loosely coupled discretizations applied to parabolic-parabolic interface problems}

\author{Erik Burman}
\address{$^1$Department of Mathematics, University College London, London, UK–WC1E 6BT, United Kingdom}
\email{e.burman@ucl.ac.uk}
\author{Rebecca Durst}
\address{$^2$Department of Mathematics, University of Pittsburgh, Pittsburgh, Pennsylvania 15261, USA}
\email{rebecca\_durst@alumni.brown.edu}
\author{Miguel A. Fern\'andez}
\address{$^3$Sorbonne Universit\'e \& CNRS, UMR 7598 LJLL, 75005 Paris France -- Inria, 75012 Paris, France}
\email{miguel.fernandez@inria.fr}
\author{Johnny Guzm\'an}
\address{$^4$Brown University}
\email{johnny\_guzman@brown.edu}
\author{Sijing Liu}
\address{$^5$The Institute for Computational and Experimental Research in Mathematics, Brown University, Providence, RI 02912, USA}
\email{sijing\_liu@brown.edu}

\begin{document}

\begin{abstract}
  We consider a parabolic-parabolic interface problem and  construct a loosely coupled prediction-correction scheme based on the Robin-Robin splitting method analyzed in {\em [J. Numer. Math., 31(1):59--77, 2023]}. We show that the errors of the correction step converge at $\mathcal O((\dt)^2)$, under suitable convergence rate assumptions on the discrete time derivative of the prediction step, where $\dt$ stands for the time-step length. Numerical results are shown to support our analysis and the assumptions. 
\end{abstract}

\maketitle


\section{Introduction}
A number of studies have been reported in the literature on loosely coupled schemes for incompressible fluid-structure interaction (FSI) with  thick-walled solids  (see, e.g., \cite{burman2009stabilization,BHS14,burman2014explicit,bukac-et-al-14,fernandez-mullaert-16,serino-et-al-19,bukac-21,gigante-vergara-21b,gigante-vergara-21,BUCELLI2023112326}). In most of these works, either no rigorous error analysis is carried out or error estimates are provided which are not uniform with respect to the spatial discretization.
Robin-Robin loosely coupled schemes, first proposed in \cite{burman2014explicit} for FSI problems, were recently analyzed in \cite{bukavc2021refactorization, burman-durst-guzman-19, burman-durst-guzman-fernandez, burman2023loosely, burman2023robin,durst-22}. Both the parabolic-parabolic and the hyperbolic-parabolic couplings have been considered in \cite{burman2023loosely}. The FSI case corresponds to the 
hyperbolic-parabolic coupling with the additional difficulty coming from the fluid incompressibility constraint \cite{bukavc2021refactorization,burman-durst-guzman-fernandez,durst-22,burman2023robin}. The method is shown to be unconditionally stable and, for the first time, with a splitting error that is spatially uniform. Sub-optimal convergence rates in time of order $\mathcal O(\sqrt{\dt})$ were first proved in \cite{bukavc2021refactorization, burman-durst-guzman-19, burman-durst-guzman-fernandez}. However,  this was later improved to nearly first-order accuracy in \cite{durst-22,burman2023loosely,burman2023robin}.

By leveraging the good stability properties and spatially uniform splitting error of the method in \cite{burman-durst-guzman-fernandez,burman2023robin}, our aim is to construct a higher-order scheme in time by building on the prediction-correction methodology \cite{bohmer1984defect},  with  the  Robin-Robin loosely coupled method as the prediction step. Although our ultimate goal is to develop a higher-order method for FSI problems, in this paper we start by constructing and analyzing a method for a coupled parabolic-parabolic problem. Indeed, the parabolic-parabolic problem is a classical model problem for fluid-fluid interaction, but it also involves some of the difficulties present in linear FSI problems, making it a good first step. It should be mentioned that correction methods have been used widely for FSI problems; see for example,  \cite{badia2008fluid,burman2009stabilization,BHS14,burman2014explicit,burman-durst-guzman-fernandez,burman2023robin}. Burman and Fern\'andez \cite{burman2009stabilization} already argued that one should expect a $\dt$ improvement for every correction step.  However, there does not seem to be an analysis of this for FSI problems (or for parabolic-parabolic problems). A method based on subcycling (several correction steps) was proposed in \cite{bene15,bene17}, with linear convergence in the number of steps. In a  domain decomposition framework, focusing on a multi-timestep approach, a Robin-Robin coupling for time-dependent advection--diffusion was introduced in \cite{CL19} with numerical investigation of the stability.

In this paper, we aim to provide rigorous evidence of this by proving second-order convergence of a prediction-correction method with a single correction step, under an assumption of the convergence of the prediction step. The guiding principle of our correction method is to have discrete time differences of the prediction step error appearing in the right-hand side of the correction step error equations. Then, one uses the fact that the discrete time differences converge with second-order accuracy.  We should mention that this strategy has also been used in \cite{aggul2018defect} for a non-linear fluid problem, and the idea goes back to the work of Stetter and B\"ohmer et al. \cite{Stett77,bohm81,bohmer1984defect}. Typically, in the previous work on defect-correction for interface problems \cite{aggul2018defect}, the physical coupling conditions in fluid-fluid interaction introduced some dissipation, which made it possible to control the splitting error.  This is often the case for fluid-fluid interactions modelling atmosphere-ocean interaction, where the two systems are coupled through the drag force on the ocean surface \cite{connors2009partitioned,CHL12,CH12,ZHS18,ZLCJ20,ZHS20,SBPK23,LHH24}. In our case, no such dissipative mechanism is present, and therefore the arguments of \cite{aggul2018defect} fail. The conservative coupling presents an important challenge for the analysis, both for the low-order splitting scheme and the defect-correction scheme. Indeed, to the best of our knowledge, this is the first time second-order convergence is proved for a splitting scheme where the coupling conditions are conservative. A key milestone towards achieving second-order convergence of the correction step was reached in an accompanying paper \cite{report}, where we prove that many of the time difference errors of the prediction step are indeed second-order accurate in the $L^2$-norm. When the interface is perpendicular to the two sides of the domain, we also prove that the $H^2$-norm of the time difference errors of the prediction step also converges as $\mathcal O((\dt)^2)$. This $H^2$ estimate turns out to be a crucial ingredient in the analysis of the correction method. Although it is nontrivial to prove the $H^2$ estimate in a more general setting, numerical results indicate that the $H^2$ estimate also holds for more general cases.

Let us summarize the steps of the prediction-correction method at time $t_{n+1}$:
\begin{itemize}
  \item Solve the prediction step to obtain the defect solutions $u_0^{n+1}$, $w_0^{n+1}$ and $\lambda_0^{n+1}$.
  \item Modify the right-hand side of the prediction step using the solutions $u_0^{n+1}$, $w_0^{n+1}$ and $\lambda_0^{n+1}$, while keeping the left-hand side of the prediction step unchanged. 
  \item The modification of the right-hand sides is designed in such a way that the right-hand sides of the error equations of the correction step consists of the term $U_0^{n+1}-U_0^{n}$, $W_0^{n+1}-W_0^{n}$, $\Lambda_0^{n+1}-\Lambda_0^{n}$ and their gradients.
  \item Solve the new system with the modified right-hand side to obtain the correction solutions $u_1^{n+1}$, $w_1^{n+1}$ and $\lambda_1^{n+1}$.
\end{itemize}
Here, $\lambda_0$ denotes the Lagrange multiplier for the prediction step. The terms $U_0^n$, $W_0^n$ and $\Lambda_0^n$ are the corresponding error terms of the variables $u_0^n$, $w_0^n$ and $\lambda_0^n$ at time $t_n$. 
As mentioned above, it has been observed numerically that $U_0^{n+1}-U_0^{n}$, $W_0^{n+1}-W_0^{n}$, and $\Lambda_0^{n+1}-\Lambda_0^{n}$ converge with rate $\mathcal O((\Delta t)^2)$ in the $L^2$-norm in general. Moreover, we prove that the first two errors are second-order accurate in the accompanying paper \cite{report}.  We also prove that the $H^2$-norm of  $U_0^{n+1}-U_0^{n}$ is second-order accurate for a special case. Note that this implies the second-order convergence of the Lagrange multiplier in $L^2$-norm.  In Assumption~\ref{assump:U0H2} below, we assume the second-order accuracy of this error in order to carry out the analysis of the correction step. With this assumption we are able to prove second-order accuracy of the correction step. It should be noted that we perform the analysis for the time semi-discrete method. Extension to the fully discrete case are non-trivial and will be considered in the future.

The rest of the paper is organized as follows. In Section \ref{sec:ppinterface}, we introduce a parabolic-parabolic interface problem and the corresponding prediction-correction method. In Section \ref{sec:eadc}, we show that the correction method gives higher-order convergence in time under a crucial lemma when the interface is perpendicular to the two sides of $\Omega$. In Section \ref{sec:numerics}, we present two numerical examples that illustrate the theoretical findings.  In Section \ref{sec:ppinterfacediri}, we consider the interface problem with Dirichlet boundary conditions and discuss an inconsistency that arises and briefly describe how to remedy it. We conclude the paper by drawing some future directions in Section \ref{sec:future}.

\section{Prediction-correction method  for a parabolic\textbackslash parabolic interface problem}\label{sec:ppinterface}

Let $\Omega$ be a bounded domain that can be decomposed as $\Omega=\Omega_f \cup \Omega_s$ where the common interface $\Sigma=\partial \Omega_f \cap \partial \Omega_s$. The subset $\Gamma_D^f$ (resp., $\Gamma_D^s$) of the boundary $\partial\Omega_f$ (resp., $\partial\Omega_s$) is the bottom side (resp., top side) of the domain $\Omega_f$ (resp., $\Omega_s$). We then denote $\Gamma_N^f=\partial\Omega_f\setminus(\Gamma_D^f\cap\Sigma)$ and $\Gamma_N^s=\partial\Omega_s\setminus(\Gamma_D^s\cap\Sigma)$. See Figure \ref{figure:neumann} for an illustration. 

 We consider the following parabolic\textbackslash parabolic interface problem.
\begin{subequations}\label{eq:ppinterface}
\begin{alignat}{2}
\pt \uu-\nu_f \Delta \uu=&0, \quad && \text{ in } [0, T] \times \Omega_f, \nonumber\\
\uu(0, x)=& \uu_0(x), \quad  && \text{ on } \Omega_f, \label{eq:ppinterfaceu}\\
\uu=&0,  \quad && \text{ on }  [0, T] \times \Gamma^f_D,\nonumber\\
\frac{\partial\uu}{\partial n}=&0,  \quad && \text{ on }  [0, T] \times \Gamma^f_N\nonumber\\
\nonumber\\
\pt \ww-\nu_s \Delta \ww=&0, \quad && \text{ in } [0, T] \times \Omega_s, \nonumber\\
\ww(0, x)=& \ww_0(x), \quad  && \text{ on } \Omega_s, \label{eq:ppinterfacew}\\
\ww=&0,  \quad && \text{ on }  [0, T] \times \Gamma^s_D,\nonumber\\
\frac{\partial\ww}{\partial n}=&0,  \quad && \text{ on }  [0, T] \times \Gamma^s_N\nonumber\\
\nonumber\\
\ww - \uu =& 0, \quad && \text{ in } [0,T] \times \Sigma, \\
\nu_s \nabla \ww \cdot \bn_s + \nu_f \nabla \uu \cdot \bn_f =& 0, \quad && \text{ in } [0,T] \times \Sigma,
\end{alignat}  
\end{subequations}
where $\bn=\bn_f$ and $\bn_s$ are the outward facing normal vectors of $\Sigma$ for $\Omega_f$ and $\Omega_s$, respectively. Here $\nu_f$ and $\nu_s$ are piece-wise constant functions. We also assume that the initial data is smooth and that $\uu$ and $\ww$ is smooth on $\Omega_f$ and $\Omega_s$ respectively.   

Note that we consider Neumann boundary conditions on two sides of the domain instead of Dirichlet boundary conditions, this is different from our previous work \cite{burman2023loosely}. There are two main reasons we consider the problem \eqref{eq:ppinterface}. First, the parabolic\textbackslash parabolic interface problem \eqref{eq:ppinterface} is more applicable in real-world scenarios, as numerous fluid-structure interaction problems involve a specific section within a broader domain where the fluid flows from left to right. Consequently, imposing Dirichlet boundary conditions to two vertical sides of the domain is unsuitable. Secondly, the numerical method for the Dirichlet problem has some inconsistencies which presents some difficulties in the analysis. We discuss this in Section \ref{sec:ppinterfacediri}.

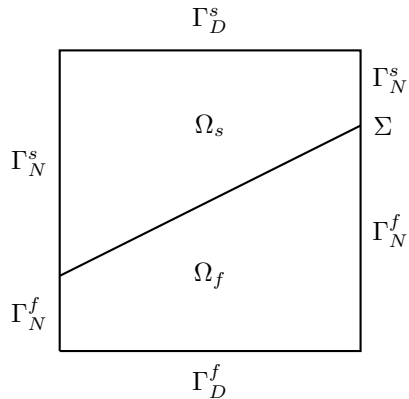
\begin{figure}[ht]
\centering
\begin{tikzpicture}
\draw[thick] plot coordinates {(0,0) (4,0) (4,4) (0,4) (0,0)};
\draw[thick] plot coordinates {(0,1) (4,3) };

 \node at (2,1) {$\Omega_f$};
 \node at (2,3) {$\Omega_s$};
 \node at (4.3,3) {$\Sigma$};
 \node at (2,-0.4) {$\Gamma^f_D$};
 \node at (2,4.4) {$\Gamma^s_D$};
 \node at (-0.4,0.5) {$\Gamma^f_N$};
 \node at (4.4,1.6) {$\Gamma^f_N$};
 \node at (-0.4,2.5) {$\Gamma^s_N$};
 \node at (4.4,3.6) {$\Gamma^s_N$};

\end{tikzpicture}
\caption{The domains $\Omega_f$ and $\Omega_s$ with interface $\Sigma$ and Neumann boundaries.} \label{figure:neumann}
\end{figure}

\subsection{Variational form}\label{sec:vf}
Let $(\cdot, \cdot)_i$ be the $L^2$-inner product on $\Omega_i$ for $i=f, s$. Moreover, let $\bl \cdot, \cdot \br$ be the $L^2$-inner product on $\Sigma$.  Consider the spaces
\begin{subequations}\label{eq:disspaces}
  \begin{alignat}{1}
 V_f:=&\{ v \in H^1(\Omega_f): v=0 \text{ on }  \Gamma^f_D \},  \\
 V_s:= &\{ v \in H^1(\Omega_s): v=0 \text{ on  } \Gamma^s_D \},  \\
V_g:= & L^2(\Sigma).
\end{alignat}
\end{subequations}
We define the discrete time derivative as
\begin{alignat*}{1}
\pdt v^{n+1}:=\frac{v^{n+1}-v^n}{\dt},
\end{alignat*}
and the discrete average in time
\begin{alignat*}{1}
v^{n+\frac12}:=\frac{v^{n+1}+v^n}{2}.
\end{alignat*}

Furthermore, define the time step $\Dt$ and an integer $N$, such that $T = N\Dt$, and  let $\uu^{n} = \uu(t_n, \cdot)$, where $t_n = n\Dt$ for $n \in \{0, 1, 2, \dots, N\}$. Define $\lla=\nu_f \nabla \uu \cdot \bn_f$ and assuming that $\lla\in L^2(\Sigma)$ we get that the solution to  \eqref{eq:ppinterface} solves, for all $t\in [0,T]$:

\begin{alignat}{2}
(\partial_t \uu, v)_f+\nu_f(\nabla \uu, \nabla v)_f- \bl \lla, v\br=&0, \quad &&\forall v \in V_f,\label{eq:uuexact}\\
(\partial_t \ww,q )_s+\nu_s (\nabla \ww, \nabla q)_s+ \bl \lla, q\br=&0 , \quad && \forall q \in V_s, \label{eq:wwexact}\\
\bl \uu-\ww, \mu \br=&0, \quad && \forall\mu \in V_g.
\end{alignat}

\subsection{Defect-correction method}
We propose the following defect-correction method. The method consists of two steps, namely, the prediction step and the correction step. The prediction step is a Robin-Robin coupling method \cite{burman2023loosely}. The correction step is designed by modifying the right-hand side of the prediction step while keeping the left-hand side of the prediction step unchanged as we discussed in the introduction.
\paragraph{\bf Prediction step}
Find $w_0^{n+1}  \in V_s$, $u_0^{n+1} \in V_f$, and $\lambda_0^{n+1} \in V_g $ such that for $n\geq 0$,
\begin{subequations}\label{defect}
\begin{alignat}{2}
(\pdt w_0^{n+1}, z)_s+\nu_s(\nabla w_0^{n+1}, \nabla z)_s + \alpha \bl  w_0^{n+1}-u_{0}^{n}, z \br + \bl \lambda_{0}^{n}, z\br=&0 , \quad && \forall z \in V_s, \label{defect1}\\
(\pdt u_0^{n+1}, v)_f+\nu_f(\nabla u_0^{n+1}, \nabla v)_f- \bl \lambda_0^{n+1}, v\br=&0, \quad &&\forall v \in V_f,  \label{defect2}\\
\bl (\lambda_0^{n+1}- \lambda_{0}^{n})+  \alpha (u_0^{n+1}-w_0^{n+1}) , \mu \br=& 0,   \quad &&\forall \mu \in V_g.  \label{defect3}
\end{alignat}
\end{subequations}

\paragraph{\bf Correction step}
Find $w_1^{n+1}  \in V_s$, $u_1^{n+1} \in V_f$, and $\lambda_1^{n+1} \in V_g $ such that for $n\geq 0$,
\begin{subequations}\label{eq:correction}
\begin{alignat}{2}
&(\pdt w_1^{n+1}, z)_s+\nu_s(\nabla w_1^{n+1}, \nabla z)_s + \alpha \bl  w_1^{n+1}-u_{1}^{n}, z \br + \bl \lambda_{1}^{n}, z\br \quad &&\label{correct1}\\
&=\nu_s(\nabla w_0^{n+1}, \nabla z)_s + \alpha \bl  w_0^{n+1}-w_{0}^{n}, z \br + \bl \lambda_{0}^{n}, z\br-\nu_s(\nabla w_0^{n+\frac12}, \nabla z)_s - \bl \lambda_{0}^{n+\frac12}, z\br, \quad &&  \forall z \in V_s,\nonumber \\
&(\pdt u_1^{n+1}, v)_f+\nu_f(\nabla u_1^{n+1}, \nabla v)_f- \bl \lambda_1^{n+1}, v\br \quad && \label{correct2}\\
&=\nu_f(\nabla u_0^{n+1}, \nabla v)_f- \bl \lambda_0^{n+1},v\br-\nu_f(\nabla u_0^{n+\frac12}, \nabla v)_f+ \bl \lambda_0^{n+\frac12}, v \br, \quad && \forall v \in V_f,  \nonumber\\
& \bl (\lambda_1^{n+1}-\lambda_{1}^{n})+ \alpha (u_1^{n+1}-w_1^{n+1}) , \mu \br = \bl \lambda_0^{n+1}-\lambda_{0}^{n}, \mu \br,  \quad &&  \forall \mu \in V_g.
\end{alignat}
\end{subequations}
In practice, the splitting method should be implemented sequentially. Assuming the information from the previous time-step is known, we solve the prediction step \eqref{defect} for $w_0^{n+1}$, $u_0^{n+1}$ and $\lambda_0^{n+1}$. We then compute the right-hand side of the correction step \eqref{eq:correction} using $w_0^{n+1}$, $u_0^{n+1}$ and $\lambda_0^{n+1}$ and solve for $w_1^{n+1}$, $u_1^{n+1}$ and $\lambda_1^{n+1}$ using \eqref{eq:correction}. Moreover, by design, each problem \eqref{defect}, \eqref{eq:correction}  can be split into two sub-problems.

\section{Estimates for the prediction step}\label{sec:predictsec}
The key to the analysis of the correction method is the error estimates of the time difference quantities of the prediction step  \eqref{defect}. The proof of these results are established in \cite{report}.  In this section, we summarize these estimates. 
We define the errors 
\begin{equation}\label{eq:errorsdef0}
  \begin{aligned}
    W_0^n&=\ww^n-w_0^n,\\
    U_0^n&=\uu^n-u_0^n,\\
    \Lambda_0^n&=\lla^n-\lambda_0^n.
  \end{aligned}
\end{equation}
In \cite[Corollary 4.2]{report}, we show the following estimates for the prediction step \eqref{defect}. 
\begin{equation}\label{eq:pdstepesti}
\begin{aligned}
    &\max_{0\le n\le N-1}\|W_0^{n+1}-W_0^n\|^2_{L^2(\Omega_s)}+\max_{0\le n\le N-1}\|U_0^{n+1}-U_0^n\|^2_{L^2(\Omega_f)}\\
    &+\dt\sum_{n=0}^{N-1}\nu_f\|\nabla(U_0^{n+1}-U_0^n)\|^2_{L^2(\Omega_f)}+\dt\sum_{n=0}^{N-1}\nu_s\|\nabla(W_0^{n+1}-W_0^n)\|^2_{L^2(\Omega_s)}\le C (\dt)^4 \mathcal{Y},
\end{aligned}
\end{equation}
Additionally, we obtain the following estimate in \cite[Corollary 4.2]{report},
  \begin{equation}\label{eq:sdesti}
    \dt \sum_{n=0}^{M-1}\nu_f  \|\frac{\nabla (U_0^{n+1}-2U_0^{n}+U_0^{n-1})}{(\dt)^2}\|_{L^2(\Omega_f)}^2\le C(\dt)^2\mathfrak{Y}.
  \end{equation}
  Here $\mathcal{Y}$ and $\mathfrak{Y}$ are defined in Appendix \ref{app:notations}.
  At last, the following relation is established in \cite{report} as well,
  \begin{equation}\label{eq:L0rela}
    \Lla_0^{n+1}=\nu_f\nabla U_0^{n+1}\cdot\bn_f \quad\text{on}\quad\Sigma.
  \end{equation}

We then make the following crucial assumption.
\begin{assumption}\label{assump:U0H2}
We have, 
  \begin{equation}\label{eq:U0H2}
  \max_{0\le n\le N-1}\|\nabla(U_0^{n+1}-U_0^n)\|^2_{L^2(\Omega_f)}+\dt\sum_{n=0}^{N-1}\|D^2(U_0^{n+1}-U_0^{n})\|^2_{L^2(\Omega_f)}\le C (\dt)^4\mathsf{Y},
  \end{equation}
  where $\mathsf{Y}$ contains the regularity of the solutions.
\end{assumption}

Due to Assumption \ref{assump:U0H2}, the following Corollary is valid.
\begin{corollary}\label{coro:lambdanew1}
We have, 
\begin{alignat}{2}
  \dt\sum_{n=0}^{N-1}\|\Lambda_0^{n+1}-\Lambda_0^n\|^2_{L^2(\Sigma)}&\le C(\dt)^4\mathsf{Y}\label{eq:lambdanew1}
\end{alignat}
for any $n=0,1,2,\ldots, N-1$.   
\end{corollary}
\begin{proof}
  By a standard trace inequality, the relation \eqref{eq:L0rela} and \eqref{eq:U0H2}, the estimate \eqref{eq:lambdanew1} immediately follows.
\end{proof}

\begin{remark}\label{rem:u0h2}
   If we consider a simple setting where the interface is perpendicular to the two sides of the domain $\Omega$, as depicted in Figure \ref{figure:neumannhori}, we can prove Assumption \ref{assump:U0H2} with $\mathsf{Y}=Y+\mathfrak{Y}+\mathbb{Y}+\nu_f\|\pt^3\ww\|^2_{L^2((0,T),L^2(\Omega_s))}$. Here $Y$ and $\mathbb{Y}$ are defined in Appendix \ref{app:notations}. We refer to \cite{report} for the proof of Assumption \ref{assump:U0H2} in this special case. However, it is nontrivial to prove Assumption \ref{assump:U0H2} for a more general case. Nonetheless, numerical results in \cite{report} indicate that Assumption \ref{assump:U0H2} holds for general cases.
\end{remark}

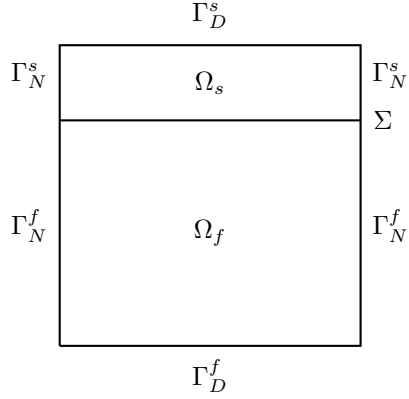
\begin{figure}[ht]
\centering
\begin{tikzpicture}
\draw[thick] plot coordinates {(0,0) (4,0) (4,4) (0,4) (0,0)};
\draw[thick] plot coordinates {(0,3) (4,3) };

 \node at (2,1.5) {$\Omega_f$};
 \node at (2,3.5) {$\Omega_s$};
 \node at (4.3,3) {$\Sigma$};
 \node at (2,-0.4) {$\Gamma^f_D$};
 \node at (2,4.4) {$\Gamma^s_D$};
 \node at (-0.4,1.6) {$\Gamma^f_N$};
 \node at (4.4,1.6) {$\Gamma^f_N$};
 \node at (-0.4,3.6) {$\Gamma^s_N$};
 \node at (4.4,3.6) {$\Gamma^s_N$};

\end{tikzpicture}
\caption{The domains $\Omega_f$ and $\Omega_s$ with horizontal interface $\Sigma$} \label{figure:neumannhori}
\end{figure}

\section{Error analysis for the correction step}\label{sec:eadc}

In this section we provide the error analysis of \eqref{eq:correction} under 
Assumption~\ref{assump:U0H2}.
As argued in \cite{burman2023loosely}, the prediction step is first-order accurate. In this section, we will prove that the correction method produces a second-order accurate method provided Assumption \ref{assump:U0H2} is satisfied.

\subsection{Error equations}
To obtain error equations, we write the exact solution in a way that can be compared to the numerical method. 

\begin{subequations}\label{eq:exact}
\begin{alignat}{2}
&(\pdt \ww^{n+1}, z)_s+\nu_s (\nabla \ww^{n+1}, \nabla z)_s + \alpha \bl  \ww^{n+1}-\uu^{n}, z \br + \bl \lla^{n}, z\br \quad &&\label{correct1}\\
&\quad=\nu_s(\nabla \ww^{n+1}, \nabla z)_s + \alpha \bl  \ww^{n+1}-\uu^{n}, z \br + \bl \lla^{n}, z\br\nonumber\\
&\quad\quad-\nu_s(\nabla \ww^{n+\frac12}, \nabla z)_s - \bl \lla^{n+\frac12}, z\br+ T_1(z), \quad &&  \forall z \in V_s, \nonumber\\
&(\pdt \uu^{n+1}, v)_f+\nu_f(\nabla \uu^{n+1}, \nabla v)_f- \bl \lla^{n+1}, v\br \quad && \label{correct2}\\
&\quad=\nu_f(\nabla \uu^{n+1}, \nabla v)_f- \bl \lla^{n+1},v\br-\nu_f(\nabla \uu^{n+\frac12}, \nabla v)_f+ \bl \lla^{n+\frac12}, v \br+ T_2(v), \quad && \forall v \in V_f,  \nonumber\\
& \bl (\lla^{n+1}-\lla^{n}) +\alpha(\uu^{n+1}-\ww^{n+1}) , \mu \br =  \bl \lla^{n+1}-\lla^{n}, \mu \br, \quad &&  \forall \mu \in V_g,
\end{alignat}
\end{subequations}
where
\begin{subequations}\label{eq:T12}
  \begin{alignat}{1}
T_1(z):= &(\pdt \ww^{n+1}, z)_s+\nu_s(\nabla \ww^{n+\frac12}, \nabla z)_s + \bl \lla^{n+\frac12}, z\br, \\
T_2(v):=& (\pdt \uu^{n+1}, v)_f+\nu_f(\nabla \uu^{n+\frac12}, \nabla v)_f - \bl \lla^{n+\frac12}, v\br.\label{eq:T2}
\end{alignat}
\end{subequations}
Now we define the errors 
\begin{equation}\label{eq:errorsdef1}
  \begin{aligned}
    W_1^n&=\ww^n-w_1^n,\\
    U_1^n&=\uu^n-u_1^n,\\
    \Lambda_1^n&=\lla^n-\lambda_1^n.
  \end{aligned}
\end{equation}
Subtract \eqref{eq:correction} and \eqref{eq:exact}, we obtain the error equations for $n\ge0$,
\begin{subequations}\label{error}
\begin{alignat}{2}
&(\pdt W_1^{n+1}, z)_s+\nu_s(\nabla W_1^{n+1}, \nabla z)_s + \alpha \bl  W_1^{n+1}-U_1^{n}, z \br + \bl \Lambda_1^{n}, z\br  =R_1(z) +T_1(z) \quad  &&  \forall z \in V_s, \label{error1}\\
&(\pdt U_1^{n+1}, v)_f+\nu_f(\nabla U_1^{n+1}, \nabla v)_f- \bl \Lambda_1^{n+1}, v\br= R_2(v)+T_2(v)
\quad && \forall v \in V_f,  \label{error2}\\
& \bl \Lambda_1^{n+1}-\Lambda_1^{n}+  \alpha(U_1^{n+1}-W_1^{n+1}), \mu \br =  \bl \Lambda_0^{n+1}-\Lambda_0^{n}, \mu \br  \quad &&  \forall\mu \in V_g.\label{error3}
\end{alignat}
\end{subequations}
Here
\begin{subequations}
  \begin{alignat}{1}
 R_1(z) :=&\frac{1}{2}\nu_s(\nabla (W_0^{n+1}-W_0^n), \nabla z)_s + \alpha \bl  (W_0^{n+1}-W_0^{n}), z \br - \frac{1}{2}\bl \Lambda_0^{n+1}-\Lambda_0^{n}, z\br,  \\
 R_2(v): =& 
\frac{1}{2} \nu_f(\nabla (U_0^{n+1}-U_0^n), \nabla v)_f - \frac{1}{2}\bl \Lambda_0^{n+1}-\Lambda_0^{n}, v\br.\label{eq:R2}
\end{alignat}
\end{subequations}
It follows from \eqref{error3} that,
\begin{equation}\label{errorstrong}
\Lambda_1^{n+1}-\Lambda_1^{n}+ \alpha ( U_1^{n+1}-W_1^{n+1})=   \Lambda_0^{n+1}-\Lambda_0^{n}\quad\text{on}\quad\Sigma.
\end{equation}
We use the convention $u_1^i=\uu^i=\uu_0$ and $w_1^i=\ww^i=\ww_0$ for $i<0$.

\subsection{Preliminary estimate for the correction step}

 In this section, we establish some preliminary estimates that will be used in the convergence analysis. The techniques we use here is similar but slightly different to those in \cite{report}. First, we define the following quantities for $n\ge0$,
\begin{alignat*}{1}
\Z_1^{n+1}:= & \frac{1}{2}\|W_1^{n+1}\|_{L^2(\Omega_s)}^2+ \frac{1}{2} \|U_1^{n+1}\|_{L^2(\Omega_f)}^2 +\frac{\dt \alpha}{2} \|U_1^{n+1}\|_{L^2(\Sigma)}^2 + \frac{\dt}{2\alpha} \|\Lambda_1^{n+1}\|_{L^2(\Sigma)}^2, \\ 
\Ss_1^{n+1}:= & \dt(\nu_f  \|\nabla U_1^{n+1}\|_{L^2(\Omega_f)}^2+  \nu_s \|\nabla W_1^{n+1}\|_{L^2(\Omega_s)}^2) +  \frac{1}{2} (\|W_1^{n+1}-W_1^n\|_{L^2(\Omega_s)}^2 + \|U_1^{n+1}-U_1^n\|_{L^2(\Omega_f)}^2)\\
&+ \frac{\alpha \dt}{2}\|(U_1^n-U_1^{n+1})+ \frac{1}{\alpha}(\Lambda_1^n-\Lambda_1^{n+1}) \|_{L^2(\Sigma)}^2.
\end{alignat*}

 We will use the following identity where $(\cdot,\cdot)$ is an inner product and $\|\cdot\|$ is the corresponding norm: 
\begin{equation}\label{eq:innerid}
  (\varphi-\psi,\vartheta)=\frac12(\|\varphi\|^2-\|\psi\|^2+\|\vartheta-\psi\|^2-\|\vartheta-\varphi\|^2).
\end{equation}
 We start with the following lemma.

\begin{lemma}\label{errorLemma1}It holds, 
\begin{alignat}{1}\label{eq:erroreq}
\Z_1^{n+1}+\Ss_1^{n+1}=\Z_1^n + \dt F_1^{n+1}  +\dt R^{n+1}+\dt T^{n+1}+\frac{\dt}{\alpha}\bl \Lambda_1^{n+1}, \Lambda_0^{n+1}-\Lambda_0^{n} \br,
\end{alignat}
where
\begin{alignat*}{1}
F_1^{n+1}:=& - \bl   W_1^{n+1}, \Lambda_0^{n}-\Lambda_0^{n+1} \br +\bl U_1^n-U_1^{n+1}, \Lambda_0^n-\Lambda_0^{n+1} \br,\\
R^{n+1}:=& R_1(W_1^{n+1})+R_2(U_1^{n+1}), \\
T^{n+1}:=& T_1(W_1^{n+1})+T_2(U_1^{n+1}).
\end{alignat*}
\end{lemma}
\begin{proof}
To begin, we set $z= \dt W_1^{n+1}$ in \eqref{error1} and $v = \dt U_1^{n+1}$ in \eqref{error2} and we obtain, by \eqref{eq:innerid},
\begin{alignat}{1}
& \frac{1}{2}\|W_1^{n+1}\|_{L^2(\Omega_s)}^2+ \frac{1}{2} \|U_1^{n+1}\|_{L^2(\Omega_f)}^2 +  \frac{1}{2}\|W_1^{n+1}-W_1^n\|_{L^2(\Omega_s)}^2+ \frac{1}{2} \|U_1^{n+1}-U_1^n\|_{L^2(\Omega_f)}^2 \nonumber  \\
 & + \nu_s\dt \|\nabla W_1^{n+1}\|_{L^2(\Omega_s)}^2+  \nu_f\dt \|\nabla U_1^{n+1}\|_{L^2(\Omega_f)}^2 \nonumber \\
  & =  \frac{1}{2}\|W_1^{n}\|_{L^2(\Omega_s)}^2+ \frac{1}{2} \|U_1^{n}\|_{L^2(\Omega_f)}^2+\dt (\J^{n+1}+R^{n+1}+T^{n+1}), \label{erroraux213}
\end{alignat}
where
\begin{alignat}{1}\label{eq:J1}
\J^{n+1}:=& - \alpha \bl W_1^{n+1} - U_1^n, W_1^{n+1} \br- \bl  \Lambda_1^n, W_1^{n+1}\br+ \bl \Lambda_1^{n+1}, U_1^{n+1}\br.
\end{alignat}

After some manipulations and use \eqref{errorstrong}, we have
\begin{alignat*}{1}
\J^{n+1}= \mathbb{J}^{n+1}  - \bl   W_1^{n+1}, \Lambda_0^{n}-\Lambda_0^{n+1} \br - \frac{1}{\alpha}\bl \Lambda_1^{n+1}, \Lambda_0^{n}- \Lambda_0^{n+1} \br  +\bl U_1^n-U_1^{n+1}, \Lambda_0^n-\Lambda_0^{n+1} \br,
\end{alignat*}
where 
\begin{equation*}
 \mathbb{J}^{n+1} := \alpha \bl U_1^n-U_1^{n+1}, U_1^{n+1} \br+ \frac{1}{\alpha} \bl  \Lambda_1^n-\Lambda_1^{n+1}, \Lambda_1^{n+1} \br- \bl U_1^n-U_1^{n+1}, \Lambda_1^n-\Lambda_1^{n+1} \br.
\end{equation*}

One can easily show that 
\begin{alignat}{1}\label{eq:JJ1}
 \mathbb{J}^{n+1}= & \frac{\alpha}{2}(\|U_1^n\|_{L^2(\Sigma)}^2 - \|U_1^{n+1}\|_{L^2(\Sigma)}^2)  + \frac{1}{2\alpha}(\|\Lambda_1^n\|_{L^2(\Sigma)}^2 - \|\Lambda_1^{n+1}\|_{L^2(\Sigma)}^2) \\
&- \frac{\alpha}{2}\|(U_1^n-U_1^{n+1})+ \frac{1}{\alpha}(\Lambda_1^n-\Lambda_1^{n+1}) \|_{L^2(\Sigma)}^2\nonumber.
\end{alignat}

Combining \eqref{eq:JJ1}, \eqref{eq:J1} and \eqref{erroraux213} we arrive at the identity \eqref{eq:erroreq}.
\end{proof}

In the following subsections, we give preliminary estimates of all the terms on the right-hand side of \eqref{eq:erroreq}. To this purpose, we will need the following trace inequality, which follows from the continuity of the trace operator and from Poincar\'e's inequality. 
\begin{proposition}
There exists constants $C_{\text{tr}}$ such that 


\begin{equation}\label{trace}
\|v\|_{L^2(\Sigma)} \le C_{\text{tr}} \|\nabla v\|_{L^2(\Omega_i)} \quad \forall v \in V_i\quad i=s,f.
\end{equation}
\end{proposition}

\subsection{Estimate for $F_1^{n+1}$}

\begin{lemma}\label{Festimate}
Let $1 \le M \le N$, then
\begin{equation}
 \dt \sum_{n=0}^{M-1} F_1^{n+1} \le  \frac{1}{8} \sum_{n=0}^{M-1} \Ss_1^{n+1}+ C D_1(M),
\end{equation}
where
\begin{alignat*}{1}
D_1(M):=\dt\sum_{n=0}^{M-1}\bigg(\frac{C_\text{tr}^2}{\nu_f}+ \frac{C_\text{tr}^2}{\nu_s}\bigg)\|\Lambda_0^{n}-\Lambda_0^{n+1}\|^2_{L^2(\Sigma)}.
\end{alignat*}
\end{lemma}

\begin{proof}
 It follows from \eqref{trace} and the definition of $\Ss_1^{n+1}$ that,
  \begin{multline*}
    \dt\sum_{n=0}^{M-1} \left(- \bl   W_1^{n+1}, \Lambda_0^{n}-\Lambda_0^{n+1} \br +\bl U_1^n-U_1^{n+1}, \Lambda_0^n-\Lambda_0^{n+1} \br\right)\\
    \le  \frac{\dt}{16}\sum_{n=0}^{M-1}\nu_s\|\nabla W_1^{n+1}\|^2_{L^2(\Omega_s)}+\frac{\dt}{16}\sum_{n=0}^{M-1}\nu_f\|\nabla U_1^{n+1}\|^2_{L^2(\Omega_f)}\\
     +C\dt\sum_{n=0}^{M-1}\bigg(\frac{C_\text{tr}^2}{\nu_f}+ \frac{C_\text{tr}^2}{\nu_s}\bigg)\|\Lambda_0^{n}-\Lambda_0^{n+1}\|^2_{L^2(\Sigma)}
    \le\frac{1}{8} \sum_{n=0}^{M-1} \Ss_1^{n+1}+ C D_1(M),
  \end{multline*}
  which completes the proof. 
\end{proof}

\subsection{Estimate for $R^{n+1}$}

\begin{lemma}\label{Restimate}
 Let $1 \le M \le N$, then
\begin{equation}\label{eq:Restimate}
 \dt \sum_{n=0}^{M-1} R^{n+1} \le  \frac{1}{8} \sum_{n=0}^{M-1} \Ss_1^{n+1}+ C D_2(M),
\end{equation}
where
\begin{alignat*}{1}
D_2(M):&=\dt\sum_{n=0}^{M-1}\bigg(\frac{C_\text{tr}^2}{\nu_f}+ \frac{C_\text{tr}^2}{\nu_s}\bigg)\|\Lambda_0^{n}-\Lambda_0^{n+1}\|^2_{L^2(\Sigma)}+\dt\sum_{n=0}^{M-1}\nu_s\|\nabla(W_0^{n+1}-W_0^n)\|^2_{L^2(\Omega_s)}\\
&\quad+\dt\sum_{n=0}^{M-1}\frac{C^2_{tr}}{\nu_s}\alpha^2\|W_0^{n+1}-W_0^n\|^2_{L^2(\Sigma)}+\dt\sum_{n=0}^{M-1}\nu_f\|\nabla(U_0^{n+1}-U_0^n)\|^2_{L^2(\Omega_f)}.
\end{alignat*}
\end{lemma}

\begin{proof}
  From \eqref{trace} and the definition of $\Ss_1^{n+1}$, it follows that 
  \begin{equation*}
  \begin{aligned}
    &\dt\sum_{n=0}^{M-1}R_1(W_1^{n+1})+R_2(U^{n+1})\\
    =&\dt\sum_{n=0}^{M-1} \frac{1}{2}\nu_s(\nabla (W_0^{n+1}-W_0^n), \nabla W_1^{n+1})_s + \alpha \bl  (W_0^{n+1}-W_0^{n}), W_1^{n+1} \br- \frac{1}{2}\bl \Lambda_0^{n+1}-\Lambda_0^{n}, W_1^{n+1}\br \\
    &+\dt\sum_{n=0}^{M-1}\frac{1}{2} \nu_f(\nabla (U_0^{n+1}-U_0^n), \nabla U_1^{n+1})_f - \frac{1}{2}\bl \Lambda_0^{n+1}-\Lambda_0^{n}, U_1^{n+1}\br\\
    \le&  \frac{\dt}{8}\sum_{n=0}^{M-1}\nu_s\|\nabla W_1^{n+1}\|^2_{L^2(\Omega_s)}+C\dt\sum_{n=0}^{M-1}\nu_s\|\nabla(W_0^{n+1}-W_0^n)\|^2_{L^2(\Omega_s)}\\
    &
    +C\dt\frac{C^2_{tr}}{\nu_s}\alpha^2\sum_{n=0}^{M-1}\|W_0^{n+1}-W_0^n\|^2_{L^2(\Sigma)}
    +C\dt\frac{C^2_{tr}}{\nu_s}\sum_{n=0}^{M-1}\|\Lambda_0^{n+1}-\Lambda_0^n\|^2_{L^2(\Sigma)}\\
    &+\frac{\dt}{8}\sum_{n=0}^{M-1}\nu_f\|\nabla U_1^{n+1}\|^2_{L^2(\Omega_f)}+C\dt\sum_{n=0}^{M-1}\nu_f\|\nabla(U_0^{n+1}-U_0^n)\|^2_{L^2(\Omega_f)}\\
    &
    +C\dt\frac{C^2_{tr}}{\nu_f}\sum_{n=0}^{M-1}\|\Lambda_0^{n+1}-\Lambda_0^n\|^2_{L^2(\Sigma)}\le\frac{1}{8} \sum_{n=0}^{M-1} \Ss_1^{n+1}+ C D_2(M),
  \end{aligned}  
  \end{equation*}
  which yields \eqref{eq:Restimate} and completes the proof. 
\end{proof}

\subsection{Estimate for $T^{n+1}$}

\begin{lemma}\label{Testimate}
Let $1 \le M \le N$, then
\begin{equation}\label{eq:Testimate}
 \dt \sum_{n=0}^{M-1} T^{n+1} \le  \frac{1}{4} \max_{1\le n \le M}  \Z_1^n +\frac18 \sum_{n=0}^{M-1} \Ss_1^{n+1}+ C D_3(M),
\end{equation}
where
\begin{alignat*}{1}
D_3(M):&=\dt T\sum_{n=0}^{M-1} \bigg(\|\pdt \ww^{n+1}-\pt\ww(t_{n+\frac12})\|^2_{L^2(\Omega_s)}+\|\pdt \uu^{n+1}-\pt\uu(t_{n+\frac12})\|^2_{L^2(\Omega_f)}\bigg)\\
&\quad+\dt \sum_{n=0}^{M-1}\bigg(\nu_s\|\nabla \ww^{n+\frac12}-\nabla \ww(t_{n+\frac12})\|^2_{L^2(\Omega_s)}+\nu_f\|\nabla \uu^{n+\frac12}-\nabla \uu(t_{n+\frac12})\|^2_{L^2(\Omega_f)}\bigg)\\
&\quad+\dt \sum_{n=0}^{M-1}\bigg(\bigg(\frac{C_\text{tr}^2}{\nu_f}+ \frac{C_\text{tr}^2}{\nu_s}\bigg)\|\lla^{n+\frac12}-\lla(t_{n+\frac12})\|^2_{L^2(\Sigma)}\bigg).
\end{alignat*}
\end{lemma}

\begin{proof}
  We have, by \eqref{eq:T12},
  \begin{equation}\label{eq:Testi}
    \begin{aligned}
      \dt \sum_{n=0}^{M-1} T^{n+1}&=\dt \sum_{n=0}^{M-1} \bigg[T_1^{n+1}(W_1^{n+1})+T_2^{n+1}(U_1^{n+1})\bigg]\\
   &=\dt \sum_{n=0}^{M-1}\bigg[(\pdt \ww^{n+1}, W_1^{n+1})_s+\nu_s(\nabla \ww^{n+\frac12}, \nabla W_1^{n+1})_s + \bl \lla^{n+\frac12}, W_1^{n+1}\br \\
&\hspace{2cm}+(\pdt \uu^{n+1}, U_1^{n+1})_f+\nu_s(\nabla \uu^{n+\frac12}, \nabla U_1^{n+1})_f - \bl \lla^{n+\frac12}, U_1^{n+1}\br\bigg].
    \end{aligned}
  \end{equation}
  It follows from \eqref{eq:wwexact} at $t_{n+\frac12}:=\frac{t_{n+1}+t_n}{2}$ that,
  \begin{equation}\label{eq:ww12w}
    (\pt\ww(t_{n+\frac12}),W_1^{n+1} )_s+\nu_s (\nabla \ww(t_{n+\frac12}), \nabla W_1^{n+1})_s+ \bl \lla(t_{n+\frac12}), W_1^{n+1}\br=0.
  \end{equation}
  Similarly, it follows from \eqref{eq:uuexact} that,
  \begin{equation}\label{eq:uu12u}
    (\partial_t \uu(t_{n+\frac12}),U_1^{n+1} )_f+\nu_f (\nabla \uu(t_{n+\frac12}), \nabla U_1^{n+1})_f- \bl \lla(t_{n+\frac12}), U_1^{n+1}\br=0.
  \end{equation}
Subtract \eqref{eq:ww12w} and \eqref{eq:uu12u} from \eqref{eq:Testi}, we obtain
\begin{equation}\label{eq:Testi1}
    \begin{aligned}
      \dt \sum_{n=0}^{M-1} T^{n+1}&=\dt \sum_{n=0}^{M-1} \bigg[T_1^{n+1}(W_1^{n+1})+T_2^{n+1}(U_1^{n+1})\bigg]\\
   =\dt \sum_{n=0}^{M-1}&\bigg[(\pdt \ww^{n+1}-\pt\ww(t_{n+\frac12}), W_1^{n+1})_s+\nu_s(\nabla \ww^{n+\frac12}-\nabla \ww(t_{n+\frac12}), \nabla W_1^{n+1})_s \\
   & + \bl \lla^{n+\frac12}-\lla(t_{n+\frac12}), W_1^{n+1}\br \\
&+(\pdt \uu^{n+1}-\pt\uu(t_{n+\frac12}), U_1^{n+1})_f+\nu_s(\nabla \uu^{n+\frac12}-\nabla \uu(t_{n+\frac12}), \nabla U_1^{n+1})_f \\
&- \bl \lla^{n+\frac12}-\lla(t_{n+\frac12}), U_1^{n+1}\br\bigg].
    \end{aligned}
  \end{equation}
Consequently, we estimate \eqref{eq:Testi1} by Cauchy-Schwarz inequality and Young's inequality as follows,
\begin{equation*}
  \begin{aligned}
    \dt \sum_{n=0}^{M-1} T^{n+1}&\le\frac{\dt}{8T}\sum_{n=0}^{M-1}\bigg(\|W_1^{n+1}\|^2_{L^2(\Omega_s)}+\|U_1^{n+1}\|^2_{L^2(\Omega_f)}\bigg)\\
    &\quad+C\dt T\sum_{n=0}^{M-1} \bigg(\|\pdt \ww^{n+1}-\pt\ww(t_{n+\frac12})\|^2_{L^2(\Omega_s)}+\|\pdt \uu^{n+1}-\pt\uu(t_{n+\frac12})\|^2_{L^2(\Omega_f)}\bigg)\\
    &\quad+\frac{\dt}{8}\sum_{n=0}^{M-1}\bigg(\nu_s\|\nabla W_1^{n+1}\|^2_{L^2(\Omega_s)}+\nu_f\|\nabla U_1^{n+1}\|^2_{L^2(\Omega_f)}\bigg)\\
    &\quad+C\dt \sum_{n=0}^{M-1}\bigg(\nu_s\|\nabla \ww^{n+\frac12}-\nabla \ww(t_{n+\frac12})\|^2_{L^2(\Omega_s)}+\nu_f\|\nabla \uu^{n+\frac12}-\nabla \uu(t_{n+\frac12})\|^2_{L^2(\Omega_f)}\bigg)\\
    &\quad+C\dt \sum_{n=0}^{M-1}\bigg(\bigg(\frac{C_\text{tr}^2}{\nu_f}+ \frac{C_\text{tr}^2}{\nu_s}\bigg)\|\lla^{n+\frac12}-\lla(t_{n+\frac12})\|^2_{L^2(\Sigma)}\bigg),
  \end{aligned}
\end{equation*}
which gives \eqref{eq:Testimate} and hence completes the proof. 
\end{proof}

\subsection{Estimate for $\frac{\dt}{\alpha}\bl \Lambda_1^{n+1}, \Lambda_0^{n+1}-\Lambda_0^{n} \br$}

In this section we estimate the term
\begin{equation}\label{eq:L1term}
  \frac{\dt}{\alpha}\sum_{n=0}^{M-1}\bl \Lambda_1^{n+1}, \Lambda_0^{n+1}-\Lambda_0^{n} \br.
\end{equation}
Note that $\nabla U_0^{n+1}\cdot\bn_f$ is well-defined on $\Omega_f$. Hence we may extend $\nabla U_0^{n+1}\cdot\bn_f$ into the domain $\Omega_f$ in a similar fashion as in \cite{report}. Let $\phi:\Omega_f\rightarrow \mathbb{R}$ 
 be an affine function that is $1$ on $\Sigma$ and vanishes on $\Gamma_D^f$. Therefore, the following properties are satisfied:
\begin{subequations}
  \begin{align}
    &\phi\in V_f,\label{eq:phiv}\\
    &0\le\phi\le 1,\label{eq:phi01}\\
    &\phi=1\quad\text{on}\quad\Sigma,\label{eq:phi1}\\
    &\|\nabla\phi\|_{L^2(\Omega_f)}\le C.\label{eq:phic}
  \end{align}
\end{subequations}
Define the  function
\begin{equation}
  \tilde{\Lla}_0^{n+1}:=\phi(\nu_f\nabla U_0^{n+1}\cdot\bn_f).
\end{equation}
Note that $\tilde{\Lla}_0^{n+1}=\Lla_0^{n+1}$ on $\Sigma$ and $\tilde{\Lla}_0^{n+1}\in V_f$, so it is indeed an extension. It then follows that
\begin{equation}
  \bl\Lambda_1^{n+1},\Lambda_0^{n+1}-\Lambda_0^{n}\br=\bl\Lambda_1^{n+1},\tilde{\Lla}_0^{n+1}-\tilde{\Lla}_0^{n}\br.
\end{equation}
Denote $\LL^{n+1}:=\Lla_0^{n+1}-\Lla_0^{n}$ and $\tilde{\LL}^{n+1}:=\tilde{\Lla}_0^{n+1}-\tilde{\Lla}_0^{n}$. Furthermore, owing to \eqref{error2} with $v=\tilde{\LL}^{n+1}$, we have
  \begin{equation}
  \begin{aligned}
    \bl\Lambda_1^{n+1},\tilde{\LL}^{n+1}\br&=(\pdt U_1^{n+1}, \tilde{\LL}^{n+1})_f+\nu_f(\nabla U_1^{n+1}, \nabla \tilde{\LL}^{n+1})_f-R_2(\tilde{\LL}^{n+1})-T_2(\tilde{\LL}^{n+1}).
  \end{aligned}  
  \end{equation}
  Therefore, we obtain
  \begin{equation}\label{eq:L1ext}
  \begin{aligned}
    \frac{\dt}{\alpha}\sum_{n=0}^{M-1}\bl\Lambda_1^{n+1},\Lambda_0^{n+1}-\Lambda_0^{n}\br&=\frac{1}{\alpha}\sum_{n=0}^{M-1}( U_1^{n+1}-U_1^n, \tilde{\LL}^{n+1})_f+\frac{\dt}{\alpha}\sum_{n=0}^{M-1}\nu_f(\nabla U_1^{n+1}, \nabla \tilde{\LL}^{n+1})_f\\
    &\quad-\frac{\dt}{\alpha}\sum_{n=0}^{M-1} R_2(\tilde{\LL}^{n+1})-\frac{\dt}{\alpha} \sum_{n=0}^{M-1}T_2(\tilde{\LL}^{n+1}).
  \end{aligned}  
  \end{equation}
  The following lemma states the estimate for $\eqref{eq:L1term}$, based on the previous identity. 
  \begin{lemma}
   Let $1 \le M \le N$, then
\begin{equation}\label{eq:lema4.6}
 \frac{\dt}{\alpha}\sum_{n=0}^{M-1}\bl\Lambda_1^{n+1},\Lambda_0^{n+1}-\Lambda_0^{n}\br\le \frac{1}{8} \sum_{n=0}^{M-1} \Ss^{n+1}+  \frac{1}{4} \max_{1\le n \le M}  \Z^n+ C D_4(M),
\end{equation}
where
\begin{equation}
\begin{aligned}
    D_4(M):=&\dt\frac{\nu_f}{\alpha^2}\sum_{n=0}^{M-1}\|\nabla \tilde{\LL}^{n+1}\|^2_{L^2(\Omega_f)}+\frac{\dt}{\alpha} \sum_{n=0}^{M-1} \|\tilde{\LL}^{n+1}\|^2_{L^2(\Omega_f)}\\
    &+\frac{ T(\dt)^3}{\alpha^2} \sum_{n=1}^{M-1} \| \frac{\tilde{\LL}^{n+1}-\tilde{\LL}^{n}}{(\dt)^2}\|_{L^2(\Omega_f)}^2+\frac{1}{\alpha^2} \|\tilde{\LL}^{M}\|_{L^2(\Omega_f)}^2\\
    &+\dt \frac{C^2_{tr}}{\nu_f}\sum_{n=0}^{M-1}\|\Lambda_0^{n+1}-\Lambda_0^n\|^2_{L^2(\Sigma)}+\dt\sum_{n=0}^{M-1}\nu_f\|\nabla(U_0^{n+1}-U_0^n)\|^2_{L^2(\Omega_f)}\\
    &+\frac{\dt}{\alpha}\sum_{n=0}^{M-1}\|\pdt \uu^{n+1}-\pt\uu(t_{n+\frac12})\|^2_{L^2(\Omega_f)}+\dt\sum_{n=0}^{M-1}\nu_f\|\nabla \uu^{n+\frac12}-\nabla \uu(t_{n+\frac12})\|^2_{L^2(\Omega_f)}\\
    &+\dt \sum_{n=0}^{M-1}\frac{C_\text{tr}^2}{\nu_f}\|\lla^{n+\frac12}-\lla(t_{n+\frac12})\|^2_{L^2(\Sigma)}.
\end{aligned}
\end{equation}
  \end{lemma}
\begin{proof}
  From \eqref{eq:L1ext}, \eqref{eq:R2} and \eqref{eq:T2}, it follows that 
  \begin{equation}\label{eq:L1ext1}
  \begin{aligned}
    &\frac{\dt}{\alpha}\sum_{n=0}^{M-1}\bl\Lambda_1^{n+1},\Lambda_0^{n+1}-\Lambda_0^{n}\br\\
    =&\frac{1}{\alpha}\sum_{n=0}^{M-1}( U_1^{n+1}-U_1^n, \tilde{\LL}^{n+1})_f+\frac{\dt}{\alpha}\sum_{n=0}^{M-1}\nu_f(\nabla U_1^{n+1}, \nabla \tilde{\LL}^{n+1})_f\\
    &-\frac{\dt}{\alpha}\sum_{n=0}^{M-1} \bigg[\frac{1}{2} \nu_f(\nabla (U_0^{n+1}-U_0^n), \nabla \tilde{\LL}^{n+1})_f - \frac{1}{2}\bl \Lambda_0^{n+1}-\Lambda_0^{n}, \tilde{\LL}^{n+1}\br\bigg]\\
    &-\frac{\dt}{\alpha} \sum_{n=0}^{M-1}\bigg[(\pdt \uu^{n+1}, \tilde{\LL}^{n+1})_f+\nu_f(\nabla \uu^{n+\frac12}, \nabla \tilde{\LL}^{n+1})_f - \bl \lla^{n+\frac12}, \tilde{\LL}^{n+1}\br\bigg].
  \end{aligned}  
  \end{equation}
  By using Cauchy-Schwarz inequality, Young's inequality and a similar argument as in Lemma \ref{Testimate}, we obtain
  \begin{equation*}\label{eq:L1ext2}
  \begin{aligned}
    \frac{\dt}{\alpha}&\sum_{n=0}^{M-1}\bl\Lambda_1^{n+1},\Lambda_0^{n+1}-\Lambda_0^{n}\br\\
    \le&\frac{1}{\alpha}\sum_{n=0}^{M-1}( U_1^{n+1}-U_1^n, \tilde{\LL}^{n+1})_f+\frac{\dt}{8}\sum_{n=0}^{M-1}\nu_f\|\nabla U_1^{n+1}\|_{L^2(\Omega_f)}^2+ C\dt\frac{\nu_f}{\alpha^2}\sum_{n=0}^{M-1}\|\nabla \tilde{\LL}^{n+1}\|^2_{L^2(\Omega_f)}\\
    &+\dt\sum_{n=0}^{M-1}\nu_f\|\nabla(U_0^{n+1}-U_0^n)\|^2_{L^2(\Omega_f)}
    +\dt \frac{C^2_{tr}}{\nu_f}\sum_{n=0}^{M-1}\|\Lambda_0^{n+1}-\Lambda_0^n\|^2_{L^2(\Sigma)}\\
    &+C\frac{\dt}{\alpha} \sum_{n=0}^{M-1} \|\tilde{\LL}^{n+1}\|^2_{L^2(\Omega_f)}
    +C\frac{\dt}{\alpha}\sum_{n=0}^{M-1}\|\pdt \uu^{n+1}-\pt\uu(t_{n+\frac12})\|^2_{L^2(\Omega_f)}\\
    & +C\dt\sum_{n=0}^{M-1}\nu_f\|\nabla \uu^{n+\frac12}-\nabla \uu(t_{n+\frac12})\|^2_{L^2(\Omega_f)}+C\dt\sum_{n=0}^{M-1}\frac{C_\text{tr}^2}{\nu_f}\|\lla^{n+\frac12}-\lla(t_{n+\frac12})\|^2_{L^2(\Sigma)}\\
     \le& \frac{1}{\alpha}\sum_{n=0}^{M-1}( U_1^{n+1}-U_1^n, \tilde{\LL}^{n+1})_f+\frac18 \sum_{n=0}^{M-1} \Ss_1^{n+1}+CD_4(M).
  \end{aligned}  
  \end{equation*}
  By a summation by parts formula, and using $U_1^0=0$, we have
  \begin{multline*}
  \frac{1}{\alpha}\sum_{n=0}^{M-1}( U_1^{n+1}-U_1^n, \tilde{\LL}^{n+1})_f
  =\frac{1}{\alpha} \sum_{n=1}^{M-1}  (U_1^n, \tilde{\LL}^{n}-\tilde{\LL}^{n+1})_f+\frac{1}{\alpha} (U_1^M, \tilde{\LL}^{M})_f- \frac{1}{\alpha} (U^0_1, \tilde{\LL}^{1})_f \\
  =-\frac{1}{ \alpha} \sum_{n=1}^{M-1}  (U_1^n, \tilde{\LL}^{n+1}-\tilde{\LL}^{n})_f+\frac{1}{\alpha} (U_1^M, \tilde{\LL}^{M})_f\\
  \le  \frac{\dt}{16 T} \sum_{n=1}^{M-1} \|U_1^n\|_{L^2(\Omega_f)}^2+  \frac{C T(\dt)^3}{\alpha^2} \sum_{n=1}^{M-1} \left\| \frac{\tilde{\LL}^{n+1}-\tilde{\LL}^{n}}{(\dt)^2}\right\|_{L^2(\Omega_f)}^2 \\
  +  \frac{1}{16} \|U_1^M\|_{L^2(\Omega_f)}^2+\frac{C}{\alpha^2} \|\tilde{\LL}^{M}\|_{L^2(\Omega_f)}^2
  \le \frac{1}{4} \max_{1\le n \le M}  \Z^n+ CD_4(M).
  \end{multline*}
  The estimate  \eqref{eq:lema4.6} then follows by combining the above two inequalities,
  which completes the proof.
\end{proof}

\subsection{Main result: error estimate}\label{sec:proofmain}

In this section, we state and prove our main result under Assumption \ref{assump:U0H2} and Corollary \ref{coro:lambdanew1}. To this purpose, we first need the following lemma.

\begin{lemma}\label{lem:Desti}
  We have, for $1\le M\le N$,
  \begin{alignat}{3}
    D_1(M)&\le C(\dt)^4\bigg(\frac{C_\text{tr}^2}{\nu_f}+ \frac{C_\text{tr}^2}{\nu_s}\bigg)\mathsf{Y},\\
    D_2(M)&\le C(\dt)^4\bigg(\frac{C_\text{tr}^2}{\nu_f}+ \frac{C_\text{tr}^2}{\nu_s}\bigg)\mathsf{Y}+\bigg(\frac{C_{tr}^4}{\nu_s^2}\alpha^2+1\bigg)\mathcal{Y},\\
    D_3(M)&\le (\dt)^4\Bigg(T\|\pt^3\ww\|^2_{L^2((0,T),L^2(\Omega_s))}+T\|\pt^3\uu\|^2_{L^2((0,T),L^2(\Omega_f))}\\
  &\quad+\nu_s\|\pt^2\ww\|^2_{L^2((0,T),H^1(\Omega_s))}+\nu_f\|\pt^2\uu\|_{L^2((0,T),H^1(\Omega_f))}\nonumber\\
  &\quad+\bigg(\frac{C_\text{tr}^2}{\nu_f}+ \frac{C_\text{tr}^2}{\nu_s}\bigg)\|\pt^2\lla\|^2_{L^2((0,T),L^2(\Sigma))}\Bigg),\nonumber\\
    D_4(M)&\le C (\dt)^4\Bigg(\bigg(\frac{\nu_f}{\alpha^2}+\frac{1}{\alpha}+1\bigg) \mathcal{Y}+\bigg(\frac{\nu_f^2}{\alpha^2}+\frac{\nu_f}{\alpha^2}+\frac{C^2_{tr}}{\nu_f}\bigg)\mathsf{Y}+\frac{T}{\alpha^2}\mathfrak{Y}\\
    &\quad+\frac{1}{\alpha}\|\pt^3\uu\|^2_{L^2((0,T),L^2(\Omega_s))}+\nu_f\|\pt^2\uu\|^2_{L^2((0,T),H^1(\Omega_s))}+\frac{C_\text{tr}^2}{\nu_f}\|\pt^2\lla\|^2_{L^2((0,T),L^2(\Sigma))}\Bigg),\nonumber
  \end{alignat}
  where $\mathsf{Y}$ is defined in Remark \ref{rem:u0h2} and $\mathcal{Y}$, $\mathfrak{Y}$ are defined in Appendix \ref{app:notations}. 
\end{lemma}

\begin{proof}
For $D_1(M)$, using Corollary \ref{coro:lambdanew1}, we have
\begin{equation}
    D_1(M)=\dt\sum_{n=0}^{M-1}\bigg(\frac{C_\text{tr}^2}{\nu_f}+ \frac{C_\text{tr}^2}{\nu_s}\bigg)\|\Lambda_0^{n}-\Lambda_0^{n+1}\|^2_{L^2(\Sigma)}\\
  \le  \bigg(\frac{C_\text{tr}^2}{\nu_f}+ \frac{C_\text{tr}^2}{\nu_s}\bigg) (\dt)^4\mathsf{Y}.
\end{equation}
From \eqref{eq:pdstepesti} and Corollary \ref{coro:lambdanew1},  it follows that 
\begin{alignat*}{1}
D_2(M)=&\dt\sum_{n=0}^{M-1}\bigg(\frac{C_\text{tr}^2}{\nu_f}+ \frac{C_\text{tr}^2}{\nu_s}\bigg)\|\Lambda_0^{n}-\Lambda_0^{n+1}\|^2_{L^2(\Sigma)}+\dt\sum_{n=0}^{M-1}\nu_s\|\nabla(W_0^{n+1}-W_0^n)\|^2_{L^2(\Omega_s)}\\
&\quad+\dt\sum_{n=0}^{M-1}\frac{C^2_{tr}}{\nu_s}\alpha^2\|W_0^{n+1}-W_0^n\|^2_{L^2(\Sigma)}+\dt\sum_{n=0}^{M-1}\nu_f\|\nabla(U_0^{n+1}-U_0^n)\|^2_{L^2(\Omega_f)}\\
&\le  \bigg(\frac{C_\text{tr}^2}{\nu_f}+ \frac{C_\text{tr}^2}{\nu_s}\bigg) (\dt)^4\mathsf{Y}+C (\dt)^4\mathcal{Y}+C\frac{C_{tr}^4}{\nu_s^2}\alpha^2 (\dt)^4\mathcal{Y}+C(\dt)^4\mathcal{Y}.
\end{alignat*}

For $D_3(M)$, we first notice that (see, e.g., \cite{burman2023loosely} for similar bounds)
  \begin{subequations}
  \begin{align}
    \sum_{n=0}^{M-1} \|\pdt \ww^{n+1}-\pt\ww(t_{n+\frac12})\|^2_{L^2(\Omega_s)}&\le \dt^3\|\pt^3\ww\|^2_{L^2((0,T),L^2(\Omega_s))},\label{eq:wcd}\\
     \sum_{n=0}^{M-1} \| \nabla(\ww^{n+1}-\ww(t_{n+\frac12}))\|^2_{L^2(\Omega_s)}&\le \dt^3\|\pt^2\ww\|^2_{L^2((0,T),H^1(\Omega_s))},\label{eq:dwcd}\\
     \sum_{n=0}^{M-1} \| \lla^{n+1}-\lla(t_{n+\frac12})\|^2_{L^2(\Sigma)}&\le \dt^3\|\pt^2\lla\|^2_{L^2((0,T),L^2(\Sigma))}.\label{eq:icd}
  \end{align}
  \end{subequations}
Hence, from these bounds, we get 
\begin{equation}\label{eq:d2uesti}
  \begin{aligned}
  D_3(M)&=\dt T\sum_{n=0}^{M-1} \bigg(\|\pdt \ww^{n+1}-\pt\ww(t_{n+\frac12})\|^2_{L^2(\Omega_s)}+\|\pdt \uu^{n+1}-\pt\uu(t_{n+\frac12})\|^2_{L^2(\Omega_f)}\bigg)\\
  &\quad+\dt \sum_{n=0}^{M-1}\bigg(\nu_s\|\nabla \ww^{n+\frac12}-\nabla \ww(t_{n+\frac12})\|^2_{L^2(\Omega_s)}+\nu_f\|\nabla \uu^{n+\frac12}-\nabla \uu(t_{n+\frac12})\|^2_{L^2(\Omega_f)}\bigg)\\
  &\quad+\dt \sum_{n=0}^{M-1}\bigg(\bigg(\frac{C_\text{tr}^2}{\nu_f}+ \frac{C_\text{tr}^2}{\nu_s}\bigg)\|\lla^{n+\frac12}-\lla(t_{n+\frac12})\|^2_{L^2(\Sigma)}\bigg)\\
  &\le  T(\dt)^4(\|\pt^3\ww\|^2_{L^2((0,T),L^2(\Omega_s))}+\|\pt^3\uu\|^2_{L^2((0,T),L^2(\Omega_f))})\\
  &\quad+(\dt)^4(\nu_s\|\pt^2\ww\|^2_{L^2((0,T),H^1(\Omega_s))}+\nu_f\|\pt^2\uu\|^2_{L^2((0,T),H^1(\Omega_f))})\\
  &\quad+(\dt)^4\bigg(\frac{C_\text{tr}^2}{\nu_f}+ \frac{C_\text{tr}^2}{\nu_s}\bigg)\|\pt^2\lla\|^2_{L^2((0,T),L^2(\Sigma))}.
    \end{aligned}
\end{equation}
For $D_4(M)$, we first estimate the terms involving $\tilde{\LL}^{n+1}$. From \eqref{eq:phi01}, \eqref{eq:phic}, \eqref{eq:pdstepesti} and Assumption~\ref{assump:U0H2}, it follows that 
\begin{equation}
  \begin{aligned}
    &\dt\frac{\nu_f}{\alpha^2}\sum_{n=0}^{M-1}\|\nabla \tilde{\LL}^{n+1}\|^2_{L^2(\Omega_f)}\\
    &\le \dt\frac{\nu_f}{\alpha^2}\sum_{n=0}^{M-1} \nu_f\|\nabla(U_0^{n+1}-U_0^{n})\|^2_{L^2(\Omega_f)}+\dt\frac{\nu_f}{\alpha^2}\sum_{n=0}^{M-1} \nu_f\|D^2(U_0^{n+1}-U_0^{n})\|^2_{L^2(\Omega_f)}\\
    &\le C (\dt)^4  \left( \frac{\nu_f}{\alpha^2}\mathcal{Y}+ \frac{\nu_f^2}{\alpha^2}\mathsf{Y}\right).
  \end{aligned}
\end{equation}
Similarly, from \eqref{eq:pdstepesti}, we have 
\begin{equation}
  \frac{\dt}{\alpha} \sum_{n=0}^{M-1} \left\|\tilde{\LL}^{n+1}\right\|^2_{L^2(\Omega_f)}\le C\frac{\dt}{\alpha} \sum_{n=0}^{M-1} \nu_f\|\nabla(U_0^{n+1}-U_0^{n})\|^2_{L^2(\Omega_f)}\le \frac{C}{\alpha}(\dt)^4\mathcal{Y}.
\end{equation}
On the other hand, from \eqref{eq:sdesti}, we have 
\begin{equation}
\begin{aligned}
  \frac{T(\dt)^3}{\alpha^2} \sum_{n=1}^{M-1} \| \frac{\tilde{\LL}^{n+1}-\tilde{\LL}^{n}}{(\dt)^2}\|_{L^2(\Omega_f)}^2&\le C\frac{T(\dt)^3}{\alpha^2} \sum_{n=1}^{M-1}\nu_f\left\|\frac{\nabla(U_0^{n+1}-2U_0^{n}+U_0^{n-1})}{(\dt)^2}\right\|_{L^2(\Omega_f)}^2\\
  &\le C\frac{T}{\alpha^2}(\dt)^4 \mathfrak{Y}.
\end{aligned}
\end{equation}
At last, from Assumption \ref{assump:U0H2}, we get 
\begin{equation}
\begin{aligned}
  \frac{1}{\alpha^2} \|\tilde{\LL}^{M}\|_{L^2(\Omega_f)}^2&\le C\frac{\nu_f}{\alpha^2}\|\nabla(U_0^{M+1}-U_0^M)\|^2_{L^2(\Omega_f)}\le C\frac{\nu_f}{\alpha^2}(\dt)^4\mathsf{Y}.
\end{aligned}
\end{equation}
The remaining terms in $D_4(M)$ can be estimated similarly to those of $D_1(M)$ to $D_3(M)$ as follows, using \eqref{eq:pdstepesti}, \eqref{eq:lambdanew1} and \eqref{eq:wcd}-\eqref{eq:icd}, yielding 
\begin{equation}
  \begin{aligned}
    &\dt \frac{C^2_{tr}}{\nu_f}\sum_{n=0}^{M-1}\|\Lambda_0^{n+1}-\Lambda_0^n\|^2_{L^2(\Sigma)}+\dt\sum_{n=0}^{M-1}\nu_f\|\nabla(U_0^{n+1}-U_0^n)\|^2_{L^2(\Omega_f)}\\
    &+\frac{\dt}{\alpha}\sum_{n=0}^{M-1}\|\pdt \uu^{n+1}-\pt\uu(t_{n+\frac12})\|^2_{L^2(\Omega_f)}+\dt\sum_{n=0}^{M-1}\nu_f\|\nabla \uu^{n+\frac12}-\nabla \uu(t_{n+\frac12})\|^2_{L^2(\Omega_f)}\\
    &+\dt \sum_{n=0}^{M-1}\frac{C_\text{tr}^2}{\nu_f}\|\lla^{n+\frac12}-\lla(t_{n+\frac12})\|^2_{L^2(\Sigma)}\\
    &\le \frac{C^2_{tr}}{\nu_f}(\dt)^4\mathsf{Y}+C(\dt)^4\mathcal{Y}+\frac{1}{\alpha}\dt^4\|\pt^3\uu\|^2_{L^2((0,T),L^2(\Omega_s))}\\
    &\quad+\nu_f\dt^4\|\pt^2\uu\|^2_{L^2((0,T),H^1(\Omega_s))}+\frac{C_\text{tr}^2}{\nu_f}\dt^4\|\pt^2\lla\|^2_{L^2((0,T),L^2(\Sigma))}.
  \end{aligned}
\end{equation}
We conclude the proof by combining all the above inequalities.
\end{proof}

We now state and prove the main result of this paper. 
\begin{theorem}\label{thm:mainthm}
The following error estimate holds  
\begin{equation}\label{eq:mainthm}
\max_{1 \le M \le N} \Z_1^{M}+ \sum_{n=0}^{N-1} \Ss_1^{n+1} \le    C (\dt)^4 \mathsf{Y}_1,
\end{equation}
where
\begin{alignat*}{1}
\mathsf{Y}_1:=&\Bigg(\frac{C_{tr}^2}{\nu_f}+\frac{C_{tr}^2}{\nu_s}+\frac{C_{tr}^2}{\nu_f^2}\alpha^2+\frac{\nu_f+\alpha}{\alpha^2}+1\Bigg)\mathcal{Y}
+\Bigg(\frac{C_{tr}^2}{\nu_f}+\frac{C_{tr}^2}{\nu_s}+\frac{\nu_f+\nu_f^2}{\alpha^2}\Bigg)\mathsf{Y}\\
& +\frac{T}{\alpha^2}\mathfrak{Y}
+T\|\pt^3\ww\|^2_{L^2((0,T),L^2(\Omega_s))}+\left(T+\frac{1}{\alpha}\right)\|\pt^3\uu\|^2_{L^2((0,T),L^2(\Omega_f))}\\
&+\nu_s\|\pt^2\ww\|^2_{L^2((0,T),H^1(\Omega_s))}+\nu_f\|\pt^2\uu\|_{L^2((0,T),H^1(\Omega_f))}+\bigg(\frac{C_\text{tr}^2}{\nu_f}+ \frac{C_\text{tr}^2}{\nu_s}\bigg)\|\pt^2\lla\|^2_{L^2((0,T),L^2(\Sigma))}.
\end{alignat*}
Here, the notation $\mathsf{Y}$ is defined in Assumption \ref{assump:U0H2} and $\mathcal{Y}$, $\mathfrak{Y}$ are defined in Appendix \ref{app:notations}.
\end{theorem}

\begin{proof}
For $1 \le M \le N$, by taking the sum over $n=0,\ldots,M-1$ in Lemma \ref{errorLemma1} and using Lemmas \ref{Festimate} and \ref{Testimate}, we hvave 
\begin{alignat*}{1}
\Z_1^{M}+ \frac{1}{2} \sum_{n=0}^{M-1} \Ss_1^{n+1} \le    \frac{1}{2} \max_{1\le n \le M}  \Z_1^n + C\sum_{i=1}^4 D_i(M).
\end{alignat*}
Taking the maximum over $1 \le M \le N$, yields 
\begin{alignat*}{1}
\frac{1}{2} \max_{1 \le M \le N} \Z^{M}\le    C\max_{1 \le M \le N} \sum_{i=1}^4 D_i(M).
\end{alignat*}
Hence, we also get
\begin{alignat*}{1}
\sum_{n=0}^{N-1} \Ss^{n+1}\le    C\max_{1 \le M \le N} \sum_{i=1}^4 D_i(M).
\end{alignat*}
The estimate \eqref{eq:mainthm} then follows from Lemma \ref{lem:Desti}, which concludes the proof. 
\end{proof}

\section{Numerical experiments}\label{sec:numerics}

In this section, we present two examples  which numerically illustrate the analysis and observations of the previous sections. We use the following notations to denote errors of the correction step,
\begin{alignat*}{2}
  &e_{u1}:= \|u_1^N-\uu^N\|_{L^2(\Omega_f)},\quad&& e_{w1}:= \|w_1^N-\ww^N\|_{L^2(\Omega_s)},\\
  &e_{du1}:= \|\nabla(u_1^N-\uu^N)\|_{L^2(\Omega_f)}, \quad && e_{\lambda}:=\|\lambda_0^N-\lla^N\|_{L^2(\Sigma)}, \\
  & e_{1,\lambda}:=\|(\lambda_0^N-\lla^N)-(\lambda_0^{N-1}-\lla^{N-1})\|_{L^2(\Sigma)}.
\end{alignat*}
All numerical experiments are performed using FEniCS and multiphenics \cite{alnaes2015fenics,multiphenics}.

\begin{example}[Non-horizontal interface]\label{ex:exneumann}

In this example, we test the  defect-correction scheme \eqref{eq:correction} in a problem with non-horizontal interface. We let $\Omega=(0,1)^2$ and we take the solution to be
\begin{equation*}
    \ww=\uu= e^{-2\pi^2 t} \cos(\pi x_1) \sin(\pi x_2).
\end{equation*}

 We let $\Sigma$ be defined as the straight line connecting $(0,0.25)$ and $(1,0.75)$. We then define $\Omega_s$ as the region above $\Sigma$ and $\Omega_f$ as the region below $\Sigma$ (see Figure~\ref{figure:neumann} for an illustration). The diffusion parameters are both set to be one, $\nu_s=\nu_f=1$. 
We choose $h=\dt$, $T=0.25$, $\alpha=4$ and use the piece-wise linear finite element method for the spatial discretization.  The results are reported in Table~\ref{fig:neumann}. As we can observe, second-order convergence $\mathcal O ((\dt)^2)$ is obtained for $e_{u1}$ and $e_{w1}$, which is in agreement with the theoretical findings provided by Theorem~\ref{thm:mainthm}. Finally, we also observe that $e_{1,\lambda}$ converges as $\mathcal O((\dt)^2)$. This matches with Corollary~\ref{coro:lambdanew1}.

\end{example}

\begin{table}[h]
\begin{center}

\begin{tabular}{|c||c|c|c|c|c|c|c|c|c|c|}
\hline
$\Dt$ &  $e_{u1}$ &  rates & $e_{w1}$ &  rates  &  $e_{\lambda}$ & rates & $e_{1,\lambda}$   & rates &$e_{du1}$ &  rates  \\
\hline
$(1/2)^2$&2.47e-02&--&1.14e-01&--&8.51e-01&--&8.54e-01&--&2.29e-01&--\\
\hline
$(1/2)^3$&1.44e-02&0.78&5.68e-02&1.01&5.01e-01&0.76&3.81e-01&1.16&7.43e-02&1.63\\
\hline
$(1/2)^4$&1.36e-02&0.08&1.35e-02&2.08&1.94e-01&1.37&1.45e-01&1.39&7.31e-02&0.02\\
\hline
$(1/2)^5$&5.89e-03&1.21&3.20e-03&2.07&5.34e-02&1.86&2.44e-02&2.58&2.82e-02&1.37\\
\hline
$(1/2)^6$&1.50e-03&1.98&1.07e-03&1.59&1.84e-02&1.54&4.38e-03&2.48&6.69e-03&2.08\\
\hline
$(1/2)^7$&3.59e-04&2.06&2.75e-04&1.95&7.49e-03&1.30&9.10e-04&2.27&1.56e-03&2.10\\
\hline
$(1/2)^8$&8.69e-05&2.05&6.84e-05&2.01&3.39e-03&1.14&2.07e-04&2.13&3.71e-04&2.07\\
\hline
$(1/2)^9$&2.13e-05&2.03&1.69e-05&2.01&1.61e-03&1.07&4.95e-05&2.07&9.01e-05&2.04\\
\hline
\end{tabular}
\end{center}
\caption{Errors and convergence rates for Example \ref{ex:exneumann}.}\label{fig:neumann}
\end{table}

\begin{example}[Different Viscosity]\label{ex:ddiffusion}
In this example, the defect-correction scheme \eqref{eq:correction} is tested with an exact solution involing different diffusion parameters in $\Omega_f$ and $\Omega_s$. We let $\Omega=(0,1)^2$, $\Sigma$ to be the straight line $x_2=0.75$ and we take the solution to be
\begin{equation}
  \ww=e^{-2\pi^2 t} \cos(\pi x_1) \sin\left(4\pi \frac{\nu_f}{\nu_s}(x_2-0.75)\right),\quad 
  \uu=e^{-2\pi^2 t} \cos(\pi x_1) \sin\big(4\pi (x_2-0.75)\big).
\end{equation}



Notice that $\ww=\uu$ and  $\Sigma$ and $\nu_s \nabla \ww \cdot \bn_s + \nu_f \nabla \uu \cdot \bn_f = 0$ on the interface $\Sigma$. We also have $\uu\in V_f$ and $\ww\in V_s$ if $\frac{\nu_f}{\nu_s}$ is an integer. We calculate  $g_1=\pt \uu-\nu_f \Delta \uu$, $g_2=\pt \ww-\nu_s \Delta \ww$ and solve the parabolic-parabolic interface problem \eqref{eq:ppinterface} with nonhomegeneous right-hand sides $g_1$ and $g_2$. 
We choose $\nu_f=2$, $\nu_s=1$, $h=\dt$, $T=0.25$, $\alpha=4$ and use the piece-wise linear finite element method for spatial discretization. We use the same notations as that of Example \ref{ex:exneumann}.

\end{example}

\begin{table}[h]
\begin{center}

\begin{tabular}{|c||c|c|c|c|c|c|c|c|c|c|}
\hline
$\Dt$ &  $e_{u1}$ &  rates & $e_{w1}$ &  rates  &  $e_{\lambda}$ & rates & $e_{1,\lambda}$   & rates &$e_{du}$ &  rates  \\
\hline
$(1/2)^2$&5.21e-01&-&3.49e-01&-&1.35e+00&-&7.90e-01&-&2.63e+00&-\\
\hline
$(1/2)^3$&1.38e-01&1.92&3.07e-01&0.19&1.23e+00&0.13&4.11e-01&0.94&1.70e+00&0.63\\
\hline
$(1/2)^4$&3.81e-02&1.86&7.14e-02&2.10&8.42e-01&0.55&9.23e-02&2.15&4.89e-01&1.79\\
\hline
$(1/2)^5$&9.85e-03&1.95&1.92e-02&1.90&4.65e-01&0.86&1.69e-02&2.45&1.27e-01&1.95\\
\hline
$(1/2)^6$&2.48e-03&1.99&4.87e-03&1.98&2.40e-01&0.96&3.89e-03&2.12&3.21e-02&1.98\\
\hline
$(1/2)^7$&6.22e-04&2.00&1.22e-03&1.99&1.22e-01&0.98&9.54e-04&2.03&8.06e-03&1.99\\
\hline
$(1/2)^8$&1.55e-04&2.00&3.06e-04&2.00&6.15e-02&0.99&2.40e-04&1.99&2.02e-03&2.00\\
\hline
$(1/2)^9$&3.89e-05&2.00&7.65e-05&2.00&3.09e-02&0.99&6.02e-05&1.99&5.05e-04&2.00\\
\hline
\end{tabular}
\end{center}
\caption{Error and convergence rates for Example \ref{ex:ddiffusion}}\label{fig:conv2}
\end{table}
We can clearly see $\mathcal O((\dt)^2)$ convergence rates for $e_{u1}$, $e_{w1}$ and $e_{du}$, which agrees with our analysis. One also observes the optimal convergence rates for $e_{1,\lambda}$ which match with Corollary \ref{coro:lambdanew1}.

\section{The Dirichlet interface problem}\label{sec:ppinterfacediri}

In this section, we briefly discuss the defect-correction methods for the following parabolic\textbackslash parabolic interface problem with Dirichlet boundary conditions on all sides (cf. \cite{burman2023loosely}). As we shall demonstrate there is an inconsistency in the formulation at the the end points of $\Sigma$ when $\Sigma$ is not perpendicular to the two sides of $\partial \Omega$. Then, this inconsistency makes the second-order rates of $\Lambda_0^{n+1}- \Lambda_0^n$ deteriorate, which means that the Corollary \ref{coro:lambdanew1} is not valid in this case. This in turn destroys the second-order convergence rate for the correction method measured in the $H^1$ norm. Nonetheless, when measured in the $L^2$ norm the approximations of the correction method seem to be second-order accurate.

The Dirichlet problem is as follows. 
\begin{subequations}\label{PPsplitF}
\begin{alignat}{2}
\pt \uu- \nu_f\Delta \uu=&0, \quad && \text{ in } (0, T) \times \Omega_f, \\
\uu(0, x)=& \uu_0(x), \quad  && \text{ in } \Omega_f, \\
\uu=&0 , \quad && \text{ on }  (0, T) \times \partial \Omega_f \setminus \Sigma,
\end{alignat}
\end{subequations}

\begin{subequations}\label{PPsplitS}
\begin{alignat}{2}
\pt \ww-\nu_s \Delta \ww=&0, \quad && \text{ in }(0, T) \times \Omega_s, \\
\ww(0, x)=& \ww_0(x), \quad  && \text{ in } \Omega_s, \\
\ww=&0,  \quad && \text{ on } (0, T) \times \partial \Omega_s \setminus \Sigma,
\end{alignat}
\end{subequations}
and
\begin{subequations}\label{PPsplitI}
\begin{alignat}{2}
\ww - \uu =& 0, \quad && \text{ in } (0, T) \times \Sigma, \\
\nu_s\nabla \ww \cdot \bn_s + \nu_f \nabla \uu \cdot \bn_f =& 0 \quad && \text{ in } (0, T) \times \Sigma,
\end{alignat}
\end{subequations}
where $\bn_f$ and $\bn_s$ are the outward facing normal vectors for $\Omega_f$ and $\Omega_s$, respectively. 

We define the following spaces:
\begin{subequations}\label{eq:disspacesdiri}
  \begin{alignat}{1}
 \tilde{V}_f=&\{ v \in H^1(\Omega_f): v=0 \text{ on }  \partial\Omega_f\setminus\Sigma \},  \\
 \tilde{V}_s= &\{ v \in H^1(\Omega_s): v=0 \text{ on  } \partial\Omega_s\setminus\Sigma \},  \\
\tilde{V}_g= & L^2(\Sigma).
\end{alignat}
\end{subequations}

Note that the variational form for \eqref{PPsplitF}-\eqref{PPsplitI} is the same as that of \eqref{eq:ppinterface}. The defect-correction method for \eqref{PPsplitF}-\eqref{PPsplitI} is almost identical to that of \eqref{eq:ppinterface}, the only difference is that we replace the $V$ spaces with the $\tilde{V}$ spaces. 

There is an inconsistency of $\lambda_0$ at the end points of $\Sigma$. As in \eqref{defect3} we have $\lambda_0^{n+1}-\lambda_0^n=\alpha(w_0^{n+1}-u_0^{n+1})$ on $\Sigma$. Moreover, $w_0^{n+1}$ and $u_0^{n+1}$ have zero Dirichlet boundary conditions at two end points of $\Sigma$ which force $\lambda_0^{n+1}=\lambda_0^n$ at those endpoints. Therefore, $\lambda_0^{n+1}$ does not change in time hence it does not approximate $\lla^{n+1}$ at those endpoints and consequently affect the convergence rates on $\Sigma$.

\begin{example}[Non-horizontal Interface for the Dirichlet parabolic-parabolic problems]\label{ex:exdiri}
We now present a numerical example which shows this phenomena.  
We let $\Omega=(0,1)^2$ and we take the solution to be
\begin{equation*}
    \ww=\uu= e^{-2\pi^2 t} \sin(\pi x_1) \sin(\pi x_2).
\end{equation*}
We use the same notations and parameters as those in Example \ref{ex:exneumann}. The convergence results are shown in Table \ref{fig:diri}.
\end{example}

\begin{table}[h]
\begin{center}

\begin{tabular}{|c||c|c|c|c|c|c|c|c|c|c|}
\hline
$\Dt$ &  $e_{u1}$ &  rates & $e_{w1}$ &  rates  &  $e_{\lambda}$ & rates & $e_{1,\lambda}$   & rates &$e_{du1}$ &  rates  \\
\hline
$(1/2)^2$&1.99e-02&-&4.18e-02&-&7.04e-01&-&6.99e-01&-&2.21e-01&-\\
\hline
$(1/2)^3$&7.58e-03&1.39&5.82e-02&-0.48&5.04e-01&0.48&1.18e-01&2.57&4.08e-02&2.43\\
\hline
$(1/2)^4$&1.40e-02&-0.88&1.18e-02&2.30&2.75e-01&0.88&1.13e-01&0.06&7.61e-02&-0.90\\
\hline
$(1/2)^5$&7.05e-03&0.99&3.55e-03&1.73&1.54e-01&0.83&2.61e-02&2.11&3.48e-02&1.13\\
\hline
$(1/2)^6$&1.83e-03&1.94&1.41e-03&1.33&1.02e-01&0.59&6.20e-03&2.07&9.32e-03&1.90\\
\hline
$(1/2)^7$&4.47e-04&2.04&3.80e-04&1.89&7.03e-02&0.54&1.84e-03&1.75&2.74e-03&1.76\\
\hline
$(1/2)^8$&1.08e-04&2.05&9.63e-05&1.98&4.94e-02&0.51&5.95e-04&1.63&1.01e-03&1.45\\
\hline
$(1/2)^9$&2.65e-05&2.03&2.41e-05&2.00&3.47e-02&0.51&2.00e-04&1.57&4.43e-04&1.19\\
\hline
\end{tabular}
\end{center}
\caption{Errors and convergence rates for Example \ref{ex:exdiri}}\label{fig:diri}
\end{table}

We can clearly see that $e_{1,\lambda}$ is not second-order accurate and hence Corollary \ref{coro:lambdanew1} is not valid in this case. Moreover, the gradient error of $\uu$ of the correction method, $e_{du1}$, is not second-order. Finally, we observe that $e_{u1}$ and $e_{w1}$ are second-order accurate. It is worthwhile mentioning that if the interface is perpendicular to  two of the sides of $\partial \Omega$ then there is no inconsistency for the Dirichlet problem and, numerically, one regains second-order convergence for all the variables above. We omit the numerical experiments here.

We finish this section by stating that there is a remedy for the Dirichlet problem, when $\nu_s=\nu_f=\nu$, in the case $\Sigma$ is straight line and that is not parallel to the two sides of $\Omega$.
To be specific, we denote $\bm{\tau}$ as the tangent vector of $\Sigma$ and $\mathbf{s}$ as the tangent vector of the side of $\Omega$. It is easy to see that there exists constants $a$ and $b$ such that $\mathbf{n}_f=a\bm{\tau}+b\mathbf{s}$, and hence $\nabla \uu\cdot\bn_f=a\nabla \uu\cdot\bm{\tau}+b\nabla \uu\cdot\mathbf{s}$. Note that $\mathbf{n}_s=-a\bm{\tau}-b\mathbf{s}$.

\begin{remark}
We take $\Omega=(0,1)^2$ and let the interface $\Sigma$ be the straight line connecting the point $(0,0.25)$ with the point $(1,0.75)$. In this case, the unit tangent to $\Sigma$ is $\bm{\tau}=[2/\sqrt{5}, 1/\sqrt{5}]^t$. The normal $\mathbf{n}_f$ is $\mathbf{n}_f=[-1/\sqrt{5}, 2/\sqrt{5}]^t$. We let $\mathbf{s}=[0,1]^t$, hence we can write $\mathbf{n}_f=\frac{-1}{2} \bm{\tau}+ \frac{\sqrt{5}}{2} \mathbf{s}$. See Figure \ref{figure:domain} for an example.
\end{remark}

\begin{figure}
\centering
\begin{tikzpicture}[scale=0.95]
\draw[thick] plot coordinates {(0,0) (4,0) (4,4) (0,4) (0,0)};
\draw[thick] plot coordinates {(0,1) (4,3) };

 \node at (2,1) {$\Omega_f$};
 \node at (2,3) {$\Omega_s$};
 \node at (4.3,3) {$\Sigma$};
 \node at (1.5,2.3) {$\mathbf{n}_f$};
 \node at (2.2,2.6) {$\mathbf{s}$};
 \node at (2.5,1.9) {$\bm{\tau}$};



 \draw[->,thick] (2,2) -- (1.8,2.4);
 \draw[->,thick,red] (2,2) -- (2.5,2.25);
 \draw[->,thick,red] (2,2) -- (2,2.6);

\end{tikzpicture}
\caption{The domains $\Omega_f$ and $\Omega_s$ with non-horizontal interface $\Sigma$.} \label{figure:domain}
\end{figure}
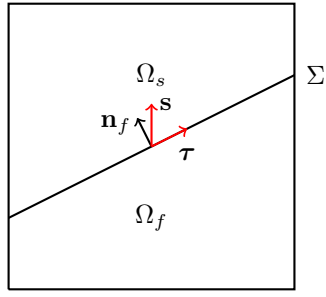

We then define the following modified defect-correction method.
\paragraph{\bf Prediction step}
Find $w_0^{n+1}  \in \tilde{V}_s$, $u_0^{n+1} \in \tilde{V}_f$, and $\lambda_0^{n+1} \in \tilde{V}_g $ such that for $n\geq 0$,
\begin{subequations}\label{defect1m}
\begin{alignat}{2}
&(\pdt w_0^{n+1}, z)_s+\nu(\nabla w_0^{n+1}, \nabla z)_s + b\alpha \bl  w_0^{n+1}-u_{0}^{n}, z \br \label{defect11m}\\
&\quad+ a\nu\bl\nabla w^{n+1}_0\cdot\bm{\tau},z\br+\bl\lambda^n,z\br=0 , \quad && \forall z \in \tilde{V}_s, \nonumber\\
&(\pdt u_0^{n+1}, v)_f+\nu(\nabla u_0^{n+1}, \nabla v)_f- \bl \lambda_{0}^{n+1}, v\br-a\nu\bl\nabla u^{n+1}_0\cdot\bm{\tau},v\br=0, \quad &&\forall v \in \tilde{V}_f,  \label{defect21m}\\
&\bl (\lambda_0^{n+1}- \lambda_{0}^{n})+  \alpha (u_0^{n+1}-w_0^{n+1}) , \mu \br= 0,   \quad &&\forall \mu \in \tilde{V}_g.  \label{defect31m}
\end{alignat}
\end{subequations}

\paragraph{\bf Correction step }
Find $w_1^{n+1}  \in \tilde{V}_s$, $u_1^{n+1} \in \tilde{V}_f$, and $\lambda_1^{n+1} \in \tilde{V}_g $ such that for $n\geq 0$,
\begin{subequations}\label{correction1m}
  \begin{alignat}{2}
&(\pdt w_1^{n+1}, z)_s+\nu(\nabla w_1^{n+1}, \nabla z)_s + b\alpha \bl  w_1^{n+1}-u_{1}^{n}, z \br\nonumber \\
&\quad\quad+ a\nu\bl\nabla w^{n+1}_1\cdot\bm{\tau},z\br+\bl \lambda_{1}^{n}, z\br \quad &&\nonumber\\
&=\nu(\nabla w_0^{n+1}, \nabla z)_s + b\alpha \bl  w_0^{n+1}-w_{0}^{n}, z \br + \bl \lambda_{0}^{n}, z\br+a\nu\bl\nabla w^{n+1}_0\cdot\bm{\tau},z\br\\
&\quad\quad-\nu(\nabla w_0^{n+\frac12}, \nabla z)_s - \bl \lambda_{0}^{n+\frac12}, z\br-a\nu\bl\nabla w^{n+\frac12}_0\cdot\bm{\tau},z\br, \quad &&  \forall z \in \tilde{V}_s,\nonumber \\
&(\pdt u_1^{n+1}, v)_f+\nu(\nabla u_1^{n+1}, \nabla v)_f- \bl \lambda_1^{n+1}, v\br -a\nu\bl\nabla u^{n+1}_1\cdot\bm{\tau},v\br\quad && \nonumber\\
&=\nu(\nabla u_0^{n+1}, \nabla v)_f- \bl \lambda_0^{n+1},v\br-a\nu\bl\nabla u^{n+1}_0\cdot\bm{\tau},v\br-\nu(\nabla u_0^{n+\frac12}, \nabla v)_f\\
&\quad\quad+ \bl \lambda_0^{n+\frac12}, v \br+a\nu\bl\nabla u^{n+\frac12}_0\cdot\bm{\tau},v\br, \quad && \forall v \in \tilde{V}_f,  \nonumber\\
& \bl (\lambda_1^{n+1}-\lambda_{1}^{n})+ \alpha (u_1^{n+1}-w_1^{n+1}) , \mu \br = \bl \lambda_0^{n+1}-\lambda_{0}^{n}, \mu \br  \quad &&  \forall \mu \in \tilde{V}_g.
\end{alignat}
\end{subequations}

\begin{remark}
  The difference between the original correction method and the modified correction method \eqref{defect1m}-\eqref{correction1m} is the definition of the Lagrange multiplier $\lambda^{n+1}$. The Lagrange multiplier $\lambda^{n+1}$ in \eqref{eq:correction} approximates $\nu\nabla \uu^{n+1}\cdot\mathbf{n}_f$ while the Lagrange multiplier $\lambda^{n+1}$ in \eqref{correction1m} approximates $b\nu\nabla\uu^{n+1}\cdot\mathbf{s}$. Note that if the interface $\Sigma$ is not perpendicular to two sides of the domain, $\lambda_0^{n+1}$ does not converge to $\nu\nabla\uu^{n+1}\cdot\mathbf{n}_f$ at those endpoints. That is why we modify the original method such that the Lagrange multiplier approximates $b\nu\nabla\uu^{n+1}\cdot\mathbf{s}$ (which is always $0$).
\end{remark}

We then show the numerical results of the modified method \eqref{correction1m} for Example \ref{ex:exdiri} in Table \ref{table:mcorrection}. We can clearly see the improvement of $e_{\lambda}$, $e_{1,\lambda}$ and $e_{du1}$ that match with our observation.

\begin{table}[h]
\begin{center}

\begin{tabular}{|c||c|c|c|c|c|c|c|c|c|c|}
\hline
$\Dt$ &  $e_{u1}$ &  rates & $e_{w1}$ &  rates  &  $e_{\lambda}$ & rates & $e_{1,\lambda}$   & rates &$e_{du1}$ &  rates  \\
\hline
$(1/2)^2$&2.20e-02&-&4.55e-02&-&7.65e-01&-&7.90e-01&-&2.31e-01&-\\
\hline
$(1/2)^3$&5.29e-03&2.06&5.43e-02&-0.26&5.19e-01&0.56&1.60e-01&2.30&3.32e-02&2.80\\
\hline
$(1/2)^4$&1.29e-02&-1.29&1.16e-02&2.22&2.33e-01&1.16&1.31e-01&0.29&7.09e-02&-1.09\\
\hline
$(1/2)^5$&7.37e-03&0.81&4.18e-03&1.48&6.57e-02&1.83&2.92e-02&2.17&3.60e-02&0.98\\
\hline
$(1/2)^6$&2.05e-03&1.84&1.63e-03&1.36&1.89e-02&1.80&4.59e-03&2.67&9.49e-03&1.92\\
\hline
$(1/2)^7$&5.09e-04&2.01&4.42e-04&1.88&6.78e-03&1.48&8.66e-04&2.41&2.32e-03&2.03\\
\hline
$(1/2)^8$&1.24e-04&2.04&1.12e-04&1.98&2.85e-03&1.25&1.85e-04&2.23&5.62e-04&2.04\\
\hline
$(1/2)^9$&3.03e-05&2.03&2.80e-05&2.00&1.31e-03&1.13&4.25e-05&2.12&1.40e-04&2.01\\
\hline
\end{tabular}
\end{center}
\caption{Error and convergence rates for modified defect-correction method \eqref{correction1m}}\label{table:mcorrection}
\end{table}

\section{Future directions}\label{sec:future}
There are several future directions. One, of course, is to prove Assumption \ref{assump:U0H2} for problems with more general interface. This seems to be a challenging problem and new ideas are needed. The second is to extend  the methodology to the wave-parabolic interface problem. Preliminary results are very promising. Once we accomplish this, we will be ready to tackle the linear FSI problem. Finally, we will consider the dynamic FSI problem.

\section*{Acknowledgement}
 This material is based upon work supported by the National Science Foundation under Grant No. DMS-1929284 while the last author was in residence at the Institute for Computational and Experimental Research in Mathematics in Providence, RI, during the "Numerical PDEs: Analysis, Algorithms, and Data Challenges" program. EB was supported by EPSRC grants EP/J002313/2 and EP/W007460/1. EB and MF were supported by the IMFIBIO Inria-UCL associated team.

\appendix

\section{Notations}\label{app:notations}

For readability, we collect various definitions in Section \ref{sec:predictsec} here.
We define $\mathcal{Y}$ as,
\begin{alignat*}{1}
\mathcal{Y}:=&  \frac{1}{\nu_s}\|\pt^3 \ww\|_{L^2(0,T;L^2(\Omega_s))}^2 + (\frac{1}{\nu_f}+1)\|\pt^3 \uu\|_{L^2(0,T;L^2(\Omega_f))}^2) \\
& + (\frac{\nu_f}{\alpha^2}+\frac{\nu_f^2}{\alpha})  \|\pt^2 \uu\|_{L^2(0,T;H^1(\Omega_f))}^2+ \frac{(\nu_f)^3}{\alpha^2}  \|\pt^2 \uu\|_{L^2(0,T;H^2(\Omega_f))}^2\\
& +\frac{\alpha^2}{\nu_s}\|\pt^2 \uu\|_{L^2(0,T;L^2(\Sigma))}^2 + \frac{1}{\nu_f}\|\pt^2 \lla\|_{L^2(0,T;L^2(\Sigma))}^2+\frac{\nu_f^2}{\alpha^2}  \|\pt^2 \uu\|_{L^\infty(0,T;H^1(\Omega_f))}^2,
\end{alignat*}
define $\mathfrak{Y}$ as,
  \begin{alignat*}{1}
\mathfrak{Y}:=&  \frac{1}{\nu_s}\|\pt^4 \ww\|_{L^2(0,T;L^2(\Omega_s))}^2 + (\frac{1}{\nu_f}+1)\|\pt^4 \uu\|_{L^2(0,T;L^2(\Omega_f))}^2) \\
& + (\frac{\nu_f}{\alpha^2}+\frac{\nu_f^2}{\alpha})  \|\pt^3 \uu\|_{L^2(0,T;H^1(\Omega_f))}^2+ \frac{(\nu_f)^3}{\alpha^2}  \|\pt^3 \uu\|_{L^2(0,T;H^2(\Omega_f))}^2\\
& +\frac{\alpha^2}{\nu_s}\|\pt^3 \uu\|_{L^2(0,T;L^2(\Sigma))}^2 + \frac{1}{\nu_f}\|\pt^3 \lla\|_{L^2(0,T;L^2(\Sigma))}^2+\frac{\nu_f^2}{\alpha^2}  \|\pt^3 \uu\|_{L^\infty(0,T;H^1(\Omega_f))}^2, 
\end{alignat*}
define $Y$ as,
\begin{equation*}
  \begin{aligned}
Y:=&  \frac{1}{\nu_s}\|\p_x\pt^2 \ww\|_{L^2(0,T;L^2(\Omega_s))}^2 + (\frac{1}{\nu_f}+1)\|\p_x\pt^2 \uu\|_{L^2(0,T;L^2(\Omega_f))}^2) \\
& + (\frac{\nu_f}{\alpha^2}+\frac{\nu_f^2}{\alpha})  \|\p_x\pt \uu\|_{L^2(0,T;H^1(\Omega_f))}^2+ \frac{(\nu_f)^3}{\alpha^2}  \|\p_x\pt \uu\|_{L^2(0,T;H^2(\Omega_f))}^2\\
& +\frac{\alpha^2}{\nu_s}\|\p_x\pt \uu\|_{L^2(0,T;L^2(\Sigma))}^2 + \frac{1}{\nu_f}\|\p_x\pt \lla\|_{L^2(0,T;L^2(\Sigma))}^2+\frac{\nu_f^2}{\alpha^2}  \|\p_x\pt \uu\|_{L^\infty(0,T;H^1(\Omega_f))}^2,
\end{aligned}
\end{equation*}
and define $\mathbb{Y}$ as,
\begin{equation*}\label{eq:bbY}
\begin{aligned}
 \mathbb{Y}:=&  \frac{1}{\nu_s}\|\p_x\pt^3 \ww\|_{L^2(0,T;L^2(\Omega_s))}^2 + (\frac{1}{\nu_f}+1)\|\p_x\pt^3 \uu\|_{L^2(0,T;L^2(\Omega_f))}^2) \\
& + (\frac{\nu_f}{\alpha^2}+\frac{\nu_f^2}{\alpha})  \|\p_x\pt^2 \uu\|_{L^2(0,T;H^1(\Omega_f))}^2+ \frac{(\nu_f)^3}{\alpha^2}  \|\p_x\pt^2 \uu\|_{L^2(0,T;H^2(\Omega_f))}^2\\
& +\frac{\alpha^2}{\nu_s}\|\p_x\pt^2 \uu\|_{L^2(0,T;L^2(\Sigma))}^2 + \frac{1}{\nu_f}\|\p_x\pt^2 \lla\|_{L^2(0,T;L^2(\Sigma))}^2+\frac{\nu_f^2}{\alpha^2}  \|\p_x\pt^2 \uu\|_{L^\infty(0,T;H^1(\Omega_f))}^2.
\end{aligned}  
\end{equation*}

\bibliographystyle{abbrv}
\bibliography{referencesWP}

\end{document}